\documentclass[11pt]{amsart}
\usepackage{amsmath,amsfonts,amssymb,mathrsfs}
\usepackage{amssymb,mathrsfs,graphicx}
\title[the Boltzmann-BGK model ]{Cauchy problem for the Boltzmann-BGK model near a global Maxwellian}
\author{Seok-Bae Yun}
\address{Department of Mathematics Sciences, KAIST (Korea Advanced Institute of Science and Technology), 373-1 Guseong-dong, Yuseung-gu, Daejeon, 305-701, Korea}
\email{sbyun@kaist.ac.kr}

\begin{document}

\newtheorem{theorem}{Theorem}[section]
\newtheorem{lemma}{Lemma}[section]
\newtheorem{corollary}{Corollary}[section]
\newtheorem{proposition}{Proposition}[section]
\newtheorem{remark}{Remark}[section]
\newtheorem{definition}{Definition}[section]

\renewcommand{\theequation}{\thesection.\arabic{equation}}
\renewcommand{\thetheorem}{\thesection.\arabic{theorem}}
\renewcommand{\thelemma}{\thesection.\arabic{lemma}}
\newcommand{\bbr}{\mathbb R}
\newcommand{\bbt}{\mathbb T}
\newcommand{\bbs}{\mathbb S}
\newcommand{\bbz}{\mathbb Z}
\newcommand{\bn}{\bf n}

\def\charf {\mbox{{\text 1}\kern-.24em {\text l}}}
\begin{abstract}
In this paper, we are interested in the Cauchy problem for the Boltzmann-BGK model for a general
class of collision frequencies.
We prove that the Boltzmann-BGK model linearized around a global Maxwellian admits a unique global smooth solution
if the initial perturbation is sufficiently small in a high order energy norm.
We also establish an asymptotic decay estimate and uniform $L^2$-stability for nonlinear perturbations.
\end{abstract}
\maketitle

\tableofcontents
%
%
%
%
\section{Introduction}
The dynamics of a monatomic, non-ionized gaseous system is known to be governed by the celebrated Boltzmann equation.
But the complicated structure of the collision operator has long been a major obstacle in developing
efficient numerical methods \cite{C}.  In an effort to find a simplified model of the Boltzmann equation,
Bhatnagar, Gross and Krook \cite{B-G-K}, and independently Walender \cite{W}, introduced the Boltzman-BGK model:
\begin{align}
\begin{aligned}\label{main.1}
\displaystyle\partial_t F  + v \cdot \nabla_x F&= \nu(\mathcal{M}(F)-F),\cr
\displaystyle F(x,v,0) &= F_0(x,v),
\end{aligned}
\end{align}
where $F(x,v,t)$ for $ (x,v,t) \in \bbt^3 \times \bbr^3 \times \bbr_+$ is the particle distribution function representing the number density of particles
in phase space at position $x$, velocity $v$ and time $t$.
$\bbt^3$ denotes the $3$-dimensional torus $\bbr^3/\bbz^3$.
$\mathcal{M}(F)$ is the local Maxwellian defined as
\[
\displaystyle\mathcal{M}(F)(x,v,t)=\frac{{\rho(x,t)}}{\sqrt{(2\pi T(x,t)})^3}\exp\Big(-\frac{|v- U(x,t)|^2}{2 T(x,t)}\Big),
\]
where $\rho$, $U$ and $T$ denote the macroscopic fields constructed from velocity moments of the distribution function:
\begin{eqnarray*}
\rho(x,t)&=& \int_{\bbr^d} F(x,v,t)dv,\\
\rho(x,t)U(x,t)&=&\int_{\bbr^d}F(x,v,t)v dv,\\
3\rho(x,t)T(x,t)&=& \int_{\bbr^d} F(x,v,t)| v-U(x,t)|^2dv.
\end{eqnarray*}
Throughout this paper, we assume that the collision frequency $\nu$ takes the following form:
\begin{eqnarray*}
\nu=\nu_{\eta,\omega}(\rho,T)\equiv \rho^{\eta}\hspace{0.1cm}T^{\omega},
\end{eqnarray*}
where we have suppressed the constant to be unity for simplicity.
A wide class of non-trivial collision frequencies is encompassed by this model. For example,
Aoki et al. \cite{A-S-Y} studied the collision frequency defined as
\begin{eqnarray*}
\nu_{1,0}=\rho.
\end{eqnarray*}
On the other hand, the following model was considered in \cite{Chap-Cowl, Miu,Y-H}:
\begin{eqnarray*}
\nu_{1,1-\omega}=\rho T^{1-\omega},
\end{eqnarray*}
where $\omega$ was chosen to be the exponent of the viscosity law of the gas.
The constant collision frequency \cite{B-G-K,Chan,F-R,Mischler,P,P-P,R-P,R-P-Y} corresponds to
\begin{eqnarray*}
\nu_{0,0}=1.
\end{eqnarray*}
\indent The relaxation operator is designed to share important features with Boltzmann collision operator. For exmaple,
the relaxation operator satisfies the following cancelation property:
\begin{eqnarray}\label{ConservationLaw}
\int_{\bbr^3} \nu\big(\mathcal{M}(F)-F\big)
\left(\begin{array}{c}
1\\
v\\
|v|^2\\
\end{array}
\right) dv=0,
\end{eqnarray}
which implies the conservation laws of mass, total momentum and total energy:
\begin{align}
\begin{aligned}\label{conservation.law.F}
&\frac{d}{dt}\int_{\bbt^3\times\bbr^3}Fdxdv=0,\cr
&\frac{d}{dt}\int_{\bbt^3\times\bbr^3}Fvdxdv=0,\cr
&\frac{d}{dt}\int_{\bbt^3\times\bbr^3}F|v|^2dxdv=0.
\end{aligned}
\end{align}
We also have the following celebrated H-theorem:
\begin{eqnarray}\label{EntropyDissipation}
\frac{d}{dt}\int_{\bbt^3\times\bbr^3} F\log Fdxdv=\int_{\bbt^3\times\bbr^3} \nu\big({\mathcal M}(F)-F\big)\log F dxdv\leq 0.
\end{eqnarray}
\indent From the numerical point of view, the BGK model considerably simplifies the situation
in that it is sufficient to update the macroscopic fields in each time step.
But mathematical analysis is not necessarily easier, because the relaxation operator
involves more nonlinearity compared to the bilinear collision operator of the Boltzmann equation.
In \cite{P}, Perthame et al. established the global existence of weak solutions for the BGK model with
constant collision frequency. Regularity and uniqueness was then considered in \cite{Iss, P-P} under the local existence
frame work. The result in \cite{Iss} was employed in \cite{R-P-Y} to prove the convergence
in a weight $L^1$ norm of a semi-Lagrangian scheme developed in \cite{F-R,R-P,Pi-Pu,Sant}, which is, as far as the author knows,
the first result on strong convergence of a fully discretized scheme for nonlinear collisional kinetic equations.
In near-a-global-Maxwellian regime,
the global existence in  the whole space $\bbr^3$ was established in \cite{Bel}
employing Ukai's spectral analysis \cite{Ukai}. Chan \cite{Chan} studied the global existence in torus
using the nonlinear energy method developed by Liu, Yang and Yu \cite{L-Y-Y}. In \cite{Chan}, however,
the decay rate is not known, which is expected to be exponentially fast.
\newline
\indent The purpose of the present paper is two-fold: first, we obtain the well-posedness
of the Boltzmann-BGK model near a global Maxwellian for a wide class of non-trivial collision frequencies.
Secondly, we establish the asymptotic decay estimate and uniform $L^2$-stability
\cite{H-Y-Y,H-L-Y-Y}.
The main theoretical tool is the nonlinear energy method developed by Guo \cite{Guo1,Guo2,Guo3} to investigate
the well-posedness of various important collisional kinetic equations such as the Boltzmann equation
or the Vlasov-Maxwell (Poisson)-Boltzmann equation. \newline
\indent Brief comments on possible extensions of our results are in order.
Our assumptions on collision frequency do not cover
the velocity dependent models proposed in \cite{Stru,Z-Stru}, which involves additional technical difficulties.
Cauchy problems for relaxation models describing ionized plasma also can be considered by extending the arguments
of this paper. We leave these topics for future research \cite{Yun}.\newline
Before we proceed further, we set some notational conventions here.
\begin{itemize}
\item $\langle\cdot,\cdot\rangle$
denotes the standard $L^2$ inner product in $\bbt^d\times\bbr^d$.
\[
\langle f,g\rangle=\int_{\bbr^3\times\bbr^3}f(x,v)g(x,v)dxdv.
\]
\item $\|\cdot\|_{L^2_{v}}$ and $\|\cdot\|_{L^2_{x,v}}$ denotes $L^2$ norms in $\bbr^d_v$
and $\bbt^d_x\times\bbr^d_v$ respectively.
\begin{eqnarray*}
&&\|f\|_{L^2_v}\equiv\Big(\int_{\bbr^3}|f(v)|^2dv\Big)^{\frac{1}{2}},\cr
&&\|f\|_{L^2_{x,v}}\equiv\Big(\int_{\bbr^3\times\bbr^3}|f(x,v)|^2dxdv\Big)^{\frac{1}{2}}.\cr
\end{eqnarray*}
\item Multi-indices $\alpha$, $\beta$ are defined by
\[
\alpha=[\alpha_0, \alpha_1, \alpha_2, \alpha_3],\quad\beta=[\beta_1, \beta_2, \beta_3]
\]
and
\[
\partial^{\alpha}_{\beta}=\partial^{\alpha_0}_{t}\partial^{\alpha_1}_{x_1}\partial^{\alpha_2}_{x_2}\partial^{\alpha_3}_{x_3}
\partial^{\beta_1}_{v_1}\partial^{\beta_2}_{v_2}\partial^{\beta_3}_{v_3}.
\]
\item The energy norm $|||\cdot|||$ is defined as follows.
\begin{eqnarray*}
|||f(t)|||\equiv\sum_{|\alpha|+|\beta|\leq N}\|\partial^{\alpha}_{\beta}f(t)\|_{L^2_{x,v}},
\end{eqnarray*}
where $N>4$.
\item We define the high order energy norm for $f$ as
\begin{equation*}
E(t)=\frac{1}{2}|||f(t)|||^2+\nu_c\int^t_0|||f(s)|||^2ds,
\end{equation*}
where the constant $\nu_c$ is defined in Proposition \ref{linearized.collision.frequency}.
\item Throughout this paper, $C_{a,b,\cdots}$ will denote a generic constant depending on $a, b, \cdots$, but
not on $x$, $v,$ and $t$.
\end{itemize}
The paper is organized as follows. In section 2, we investigate the linearization
procedure of the Boltzmann-BGK model. In section 3, the main theorem is stated.
In section 4, we present several important technical lemmas.
Section 5 is devoted to establishing the local in time existence and uniqueness of smooth solutions.
In section 6, we study the coercive property of the linearized relaxation operator.
Finally, in section 7, we combine these results to obtain the global in time existence of the classical solution.
%
%
%
%
\section{Linearized BGK model}
In this section, we consider the linearization of the Boltzmann-BGK model around the normalized global Maxwellian:
\[
m(v)=\frac{1}{\sqrt{(2\pi)^3}}e^{-\frac{|v|^2}{2}}.
\]
We first establish a technical lemma which will be frequently used in the sequel.
%
%
%
\begin{lemma}\label{Jacobian}
Let $G$ be a function given by
\[
G(x,t)=\frac{\rho|U|^2+3\rho T}{\sqrt{6}}-\frac{3\rho}{\sqrt{6}}.
\]
Then the Jacobian matrix of the change of variable $(\rho, U, T)\rightarrow (\rho, \rho U,G)$ is given by
\begin{eqnarray*}
\frac{\partial(\rho,\rho U, G)}{\partial(\rho, U, T)}=
\left(\begin{array}{ccccc}
1&0&0&0&0\cr
u_1&\rho&0&0&0\cr
u_2&0&\rho&0&0\cr
u_3&0&0&\rho&0\cr
\frac{|U|^2+3T-3}{\sqrt{6}}&\frac{2\rho U_1}{\sqrt{6}}&\frac{2\rho U_2}{\sqrt{6}}&\frac{3\rho U_3}{\sqrt{6}}&
\frac{3\rho }{\sqrt{6}}
\end{array}\right)
\end{eqnarray*}
and
\begin{eqnarray*}
\Big(\frac{\partial(\rho,\rho U, T)}{\partial(\rho,\rho U, G)}\Big)^{-1}=
\left(\begin{array}{ccccc}
1&0&0&0&0\cr
-\frac{U_1}{\rho}&\frac{1}{\rho}&0&0&0\cr
-\frac{U_2}{\rho}&0&\frac{1}{\rho}&0&0\cr
-\frac{U_3}{\rho}&0&0&\frac{1}{\rho}&0\cr
A&B_1&B_2&B_3&C
\end{array}\right),
\end{eqnarray*}
where
\begin{eqnarray*}
A&=&\frac{2|U|^2-(|U|^2+3T-3)\Big[1+\frac{U_1+U_2+U_3}{\rho}\Big]}{3\rho-(|U|^2+3T-3)},\cr
B_i&=&-\frac{2\rho U_i-(|U|^2+3T-3)}{\rho(3\rho-(|U|^2+3T-3))},\cr
C&=&\frac{\sqrt{6}}{3\rho-(|U|^2+3T-3)}.
\end{eqnarray*}
\end{lemma}
\begin{proof}
It can be verified by a straightforward, but very tedious and lengthy calculation. We omit the proof.
\end{proof}
%
%
%
%
Before we proceed to the next lemma, we define an operator which plays an important role in the theory of
kinetic equations:
\begin{definition}\label{macroscopic.projection}
The macroscopic projection is defined by
\[
Pf\equiv\sum^5_{i=1}\langle f,e_i\rangle e_i,
\]
where $\{e_i\}$ is an orthonormal basis for five- dimensional linear space
spanned by $\{\sqrt{m},v_1\sqrt{m},$\newline$v_2\sqrt{m},v_3\sqrt{m}, |v|^2\sqrt{m}\}$:
\begin{eqnarray}
\left\{\begin{array}{l}
e_1=\sqrt{m},\cr
e_2=v_1\sqrt{m},\cr
e_3=v_2\sqrt{m},\cr
e_4=v_3\sqrt{m},\cr
e_5=\frac{|v|^2-3}{\sqrt{6}}\sqrt{m}.
\end{array}\right.
\end{eqnarray}
\end{definition}
\begin{proposition}\label{Linearized_Maxwellian}
Let $F=m+\sqrt{m}f$. Then the local Maxwellian $\mathcal{M}(F)$ can be linearized around a global Maxwellian $m$ as follows
\begin{eqnarray*}
\mathcal{M}(F)=m+Pf\sqrt{m}+\sum_{1\leq i,j\leq 3}\Big(\int^1_0\big\{D^2_{(\rho_{\theta},\rho_{\theta} U_{\theta}, G_{\theta})}\mathcal{M}(\theta)\big\}(1-\theta)^2d\theta\Big)
\langle f,e_i\rangle\langle f,e_j\rangle.
\end{eqnarray*}
Here $\mathcal{M}(\theta)$ denotes
\begin{eqnarray*}
\mathcal{M}(\theta)=\frac{\rho_{\theta}}{\sqrt{(2\pi T_{\theta})^3}}e^{-\frac{|v-U_{\theta}|^2}{2T_{\theta}}},
\end{eqnarray*}
where $\rho_{\theta}, U_{\theta}, T_{\theta}$ are defined by the following relations:
\begin{align}
\begin{aligned}\label{Macro_Theta}
&\rho_{\theta}=\theta\rho+(1-\theta)1,\cr
&\rho_{\theta} U_{\theta}=\theta \rho U,\cr
&\frac{\rho_{\theta}|U_{\theta}|^2+3\rho_{\theta} T_{\theta}}{2}-\frac{3}{2}\rho_{\theta}
=\theta\Big\{\frac{\rho|U|^2+3\rho T}{2}-\frac{3}{2}\rho\Big\}.
\end{aligned}
\end{align}
\end{proposition}
\begin{proof}
We define $f(\theta)$ as follows
\begin{eqnarray*}
f(\theta)&\equiv&\mathcal{M}\Big(\theta\Big(~\rho,~\rho U,~\frac{\rho|U|^2+3\rho T}{\sqrt{6}}-\frac{3\rho}{\sqrt{6}}~\Big)+(1-\theta)(1,0,0)\Big)\cr
&\equiv&\mathcal{M}\big(\rho_{\theta},\rho_{\theta}U_{\theta},
G_{\theta}\big)\cr
&\equiv&\mathcal{M}(\theta).
\end{eqnarray*}
We note that $f$ represents the transition from the global Maxwellian $m$ to the local Maxwellian $\mathcal{M}$:
\begin{eqnarray}\label{transition}
f(1)=\frac{\rho}{\sqrt{(2\pi T)^3}}e^{-\frac{|v-U|^2}{2T}}~\mbox{ and }~f(0)=\frac{1}{\sqrt{(2\pi)^3}}e^{-\frac{|v|^2}{2}}.
\end{eqnarray}
We then apply Taylor's theorem around $\theta=0$  to see
\begin{eqnarray}\label{TaylorExpansion}
f(1)=f(0)+f^{\prime}(0)+\int^1_0f^{\prime\prime}(\theta)(1-\theta)^2d\theta.
\end{eqnarray}
(i) $f^{\prime}(0)$: We have from Lemma \ref{Jacobian} and the chain rule
\begin{align}
\begin{aligned}\label{fprime}\nonumber
f^{\prime}(0)&=\frac{d}{d\theta}\mathcal{M}(\theta(\rho,\rho U,G)+(1-\theta)(1,0,0))\Big|_{\theta=0}\cr
&=(\rho-1,\rho U,G) D_{(\rho_{\theta},\rho_{\theta} U_{\theta}, G_{\theta})}\mathcal{M}(\theta(\rho,\rho U,G)-(1-\theta)(1,0,0))\Big|_{\theta=0}\cr
&=(\rho-1,\rho U,G)\cdot\Big(\frac{\partial(\rho_{\theta},\rho_{\theta} U_{\theta},G_{\theta})}{\partial(\rho_{\theta},\rho_{\theta}, T_{\theta})}\Big)^{-1}
\Big(\nabla_{(\rho_{\theta},U_{\theta},T_{\theta})}\mathcal{M}\Big)^T\Big|_{\theta=0}\cr
&=(\rho-1,\rho U,G)\left(\begin{array}{ccccc}
1&0&0&0&0\\
0&1&0&0&0\\
0&0&1&0&0\\
0&0&0&1&0\\
0&0&0&0&\frac{\sqrt{6}}{3}
\end{array}
\right)
\left(\begin{array}{c}
1\cr
v_1\cr
v_2\cr
v_3\cr
\frac{|v|^2-3}{2}
\end{array}
\right)m\cr
&=(\rho-1) m+(\rho U) vm+G\frac{|v|^2-3}{\sqrt{6}}m\cr
&=\Big(\int f \sqrt{m}dv\Big) m+\Big(\int f \sqrt{m}dv\Big) vm+\Big(\int f\frac{|v|^2-3}{\sqrt{6}}\sqrt{m}dv\Big)\frac{|v|^2-3}{\sqrt{6}}m\cr
&=Pf\sqrt{m}.
\end{aligned}
\end{align}
(ii) $\displaystyle\int^1_0f^{\prime\prime}(\theta)(1-\theta)^2d\theta$: We have from the chain rule
\begin{eqnarray*}
f^{\prime\prime}(\theta)&=&\frac{d^2\mathcal{M}}{d\theta^2}(\theta(\rho-1,\rho U,G)-(1-\theta)(1,0,0))\cr
&=&(\rho-1,\rho U, G)\Big\{D^2_{(\rho_{\theta},\rho_{\theta} U_{\theta}, G_{\theta})}\mathcal{M}(\theta)\Big\}(\rho-1,\rho U, G)^T.
\end{eqnarray*}
We then substitute (i) and (ii) into (\ref{TaylorExpansion}) to obtain the desired result.
\end{proof}
We now consider the linearization of the collision frequency.
%
%
%
%
\begin{proposition}\label{linearized.collision.frequency}
The collision frequency can be linearized around the normalized global Maxwellian as follows.
\begin{eqnarray*}
\nu=\nu_c+\nu_p,
\end{eqnarray*}
where
\begin{eqnarray*}
\nu_c=\Big(\frac{3}{2}\Big)^{\omega}\mbox{ and }~\nu_p=\sum_i\langle f,e_i\rangle\int^1_0D_{(\rho_{\theta},\rho_{\theta} U_{\theta},G_{\theta})}\nu(\theta)(1-\theta)d\theta.
\end{eqnarray*}
Here $\nu(\theta)$ denotes
\[
\nu(\theta)=\rho^{\eta}_{\theta}\hspace{0.1cm}T^{\omega}_{\theta},
\]
where $\rho_{\theta}, U_{\theta}, T_{\theta}$ are defined as in the previous proposition.
\end{proposition}
\begin{proof}
Since the proof is almost identical to the previous one, We omit it.
\end{proof}
Since the exact form of $D^2_{(\rho_{\theta},\rho_{\theta} U_{\theta}, G_{\theta})}\mathcal{M}(\theta)$ is too complicated to be written down and
manipulated explicitly, we introduce
generic notations which considerably simplifies the argument. We first note from the chain rule
\begin{eqnarray*}
D^2_{(\rho,\rho U, G)}\mathcal{M}(\theta)&=&\nabla_{(\rho_{\theta},\rho_{\theta} U_{\theta},G_{\theta})}
\left(
\begin{array}{c}
\frac{1}{\rho_{\theta}}\cr
\frac{v-U_{\theta1}}{\rho_{\theta} T_{\theta}}-\frac{U_{\theta1}}{\rho^2_{\theta}}\cr
\frac{v-U_{\theta2}}{\rho_{\theta} T_{\theta}}-\frac{U_{\theta2}}{\rho^2_{\theta}}\cr
\frac{v-U_{\theta3}}{\rho_{\theta} T_{\theta}}-\frac{U_{\theta3}}{\rho^2_{\theta}}\cr
\frac{A}{\rho_{\theta}}+B\cdot \frac{v-T_{\theta}}{T_{\theta}}+C\frac{|v-U_{\theta}|^2-3T_{\theta}}{2T^2_{\theta}}
\end{array}
\right)\mathcal{M}(\theta)\cr\newline\newline
&=&\Big(\frac{\partial(\rho_{\theta}, \rho_{\theta} U_{\theta}, G_{\theta})}{\partial(\rho_{\theta}, U_{\theta}, T_{\theta})}\Big)^{-1}\left(
\begin{array}{c}
\nabla_{(\rho_{\theta},U_{\theta},T_{\theta})}\Big(\frac{1}{\rho_{\theta}}\Big)\cr
\nabla_{(\rho_{\theta},U_{\theta},T_{\theta})}\Big(\frac{v-U_{\theta1}}{\rho_{\theta} T}-\frac{U_{\theta 1}}{\rho^2_{\theta}}\Big)\cr
\nabla_{(\rho_{\theta},U_{\theta},T_{\theta})}\Big(\frac{v-U_{\theta2}}{\rho_{\theta} T_{\theta}}-\frac{U_{\theta2}}{\rho^2_{\theta}}\Big)\cr
\nabla_{(\rho_{\theta},U_{\theta},T_{\theta})}\Big(\frac{v-U_{\theta3}}{\rho_{\theta} T_{\theta}}-\frac{U_{\theta3}}{\rho^2_{\theta}}\Big)\cr
\nabla_{(\rho_{\theta},U_{\theta},T_{\theta})}\Big(\frac{A}{\rho_{\theta}}+B\cdot \frac{v-T_{\theta}}{T_{\theta}}+C\frac{|v-U_{\theta}|^2-3T_{\theta}}{2T_{\theta}^2}\Big)
\end{array}
\right)\mathcal{M}(\theta),
\end{eqnarray*}
where $A$, $B$, $C$ are rational functions of macroscopic fields defined in Lemma \ref{Jacobian}.
Therefore, we can deduce from Lemma \ref{Jacobian} that there exist polynomials $P^{\mathcal{M}}_{i,j}$, $R^{\mathcal{M}}_{i,j}$ such that
\begin{eqnarray*}
D^2_{(\rho_{\theta},\rho_{\theta} U_{\theta}, G_{\theta})}\mathcal{M}(\theta)
=\sum_{i,j}\frac{P^{\mathcal{M}}_{ij}(\rho_{\theta},v-U_{\theta},U_{\theta},T_{\theta})}{R^{\mathcal{M}}_{ij}(\rho_{\theta},
T_{\theta},\mathcal{G}_{\theta})}
e^{-\frac{|v-U_{\theta}|^2}{2T_{\theta}}},
\end{eqnarray*}
where
\[
\mathcal{G}_{\theta}\equiv3+3\rho_{\theta}-|U|_{\theta}^2-3T_{\theta}\rho_{\theta}
\]
and $P^{\mathcal{M}}_{ij}(x_1,\ldots ,x_n)$ and $R^{\mathcal{M}}_{ij}(x_1,\ldots,x_n)$ satisfy the following structural assumptions
${\bf(\mathcal{H}_{\mathcal{M}}})$:
\begin{itemize}
\item $(\mathcal{H}_{\mathcal{M}}1)$ $P^{\mathcal{M}}_{ij}$ is a polynomial such that $P_{ij}(0,0,\cdots, 0)=0.$
\item $(\mathcal{H}_{\mathcal{M}}2)$ $R^{\mathcal{M}}_{ij}$ is a monomial.
\end{itemize}
More precisely, we have for a multi-index $m=(m_1,m_2,\cdots,m_3)$
\begin{itemize}
\item $(\mathcal{H}_{\mathcal{M}}1)$~$P^{\mathcal{M}}_{ij}(x_1,\ldots ,x_n)=\sum_m a_{m}x^{m_1}_1x^{m_2}_2\cdots x^{m_n}_n, \mbox{ where }a_0=0$,
\item $(\mathcal{H}_{\mathcal{M}}2)$~$R^{\mathcal{M}}_{ij}(x_1,\ldots ,x_n)= a_mx^{m_1}_1x^{m_2}_2\cdots x^{m_n}_n$.
\end{itemize}
From now on, we assume that $P^{\mathcal{M}}_{ij}$ and  $R^{\mathcal{M}}_{ij}$ are defined generically,
which means the exact form may change from line to line. These generic notations simplify the
calculation drastically and cause no problems if we only keep in mind the structural assumption $(\mathcal{H}_{\mathcal{M}})$ at each step.
We now  simplify the notation further by defining
\[
Q^{\mathcal{M}}_{i,j}=\frac{P^{\mathcal{M}}_{ij}(\rho_{\theta},U_{\theta},T_{\theta},v-U_{\theta})}{R^{\mathcal{M}}_{ij}(\rho_{\theta},T_{\theta},\mathcal{G}_{\theta})}
\]
and
\begin{eqnarray*}
\mathcal{Q}^{\mathcal{M}}_{i,j}
=\int^1_0Q^{\mathcal{M}}_{ij}e^{-\frac{|v-T_{\theta}|^2}{2T_{\theta}}+\frac{|v|^2}{4}}(1-\theta)^2d\theta
\end{eqnarray*}
to see from Proposition \ref{Linearized_Maxwellian}
\begin{eqnarray}\label{Mf=}
\mathcal{M}(F)=m+\sqrt{m}Pf+\sqrt{m}\sum_{i,j}\mathcal{Q}^{\mathcal{M}}_{i,j}\langle f,e_i\rangle\langle f,e_j\rangle.
\end{eqnarray}
By the exactly same argument, we write the collision frequency as follows:
\begin{eqnarray}\label{nu=}
\nu=\nu_c+\sum_i\mathcal{Q}^{\nu}_i\langle f,e_i\rangle,
\end{eqnarray}
where
\begin{eqnarray*}
\mathcal{Q}^{\nu}_i&\equiv&\sum_i\langle f,e_i\rangle\int^1_0 Q^{\nu}_i(1-\theta)d\theta\cr
&\equiv&\sum_i\langle f,e_i\rangle\int^1_0D_{(\rho_{\theta},\rho_{\theta} U_{\theta},G_{\theta})}\nu(\theta)(1-\theta)d\theta.
\end{eqnarray*}
We again introduce generic polynomials $ P^{\nu}_{i}(x_1,\ldots ,x_n)$ and $R^{\nu}_{i}(x_1,\ldots,x_n)$ such that:
\[
D_{(\rho_{\theta},\rho_{\theta} U_{\theta},G_{\theta})}\nu(\theta)=\frac{P^{\nu}_{ij}(\rho_{\theta},U_{\theta},T_{\theta})}{R^{\nu}_{i}(\rho_{\theta},T_{\theta},\mathcal{G}_{\theta})}
\]
and assume they satisfy the same structural assumptions.
\begin{align}
\begin{aligned}\label{P.Q2}
(\mathcal{H}_{\nu}1):~P^{\nu}_{i}(x_1,\ldots ,x_n)&=\sum_m a_{m}x^{m_1}_1x^{m_2}_2\cdots x^{m_n}_n, \mbox{ where }a_0\neq0,\cr
(\mathcal{H}_{\nu}2):~R^{\nu}_{i}(x_1,\ldots ,x_n)&= Cx^{m_1}_1x^{m_2}_2\cdots x^{m_n}_n.
\end{aligned}
\end{align}
We summarize the argument so far in the following proposition.
\begin{proposition}\label{linearized.relaxation}The relaxation operator is linearized around $m$ as follows.
\begin{eqnarray*}
\frac{\nu}{\sqrt{m}}\big(\mathcal{M}(F)-F\big)=(\nu_c+\nu_p)\Big\{ \big(Pf-f\big)+\sum_{i,j}\mathcal{Q}^{\mathcal{M}}_{i,j}\langle f,e_i\rangle\langle f,e_j\rangle\Big\}.
\end{eqnarray*}
\end{proposition}
We now substitute the standard perturbation $F=m+\sqrt{m}f$ into (\ref{main.1})
and apply Proposition \ref{Linearized_Maxwellian}, \ref{linearized.collision.frequency} and
\ref{linearized.relaxation} to obtain the perturbed Boltzmann-BGK model:
\begin{eqnarray*}\label{Linearized_Main}
\displaystyle\frac{\partial f}{\partial t}  + v \cdot \nabla_x f&=&Lf+\Gamma f,\cr
\displaystyle f(x,v,0)&=&f_0(x,v), \qquad  (x,v,t) \in \bbt^3 \times \bbr^3 \times \bbr_+,
\end{eqnarray*}
where $f_0(x,v)=\frac{F_0-m}{\sqrt{m}}$. The linearized relaxation operator $L$ is given by
\begin{eqnarray}
Lf&=&\nu_c\big(Pf-f\big),
\end{eqnarray}
and the nonlinear perturbation $\Gamma(f)$ is defined as follows:
\begin{eqnarray*}
\Gamma (f)&=&\nu_p Lf+(\nu_c+\nu_p)\sum_{1\leq i,j\leq 5}\mathcal{Q}^{\mathcal{M}}_{i,j}\langle f,e_i\rangle \langle f,e_j\rangle,\cr
&=&\Gamma_1(f,f)-\Gamma_2(f,f)+\Gamma_3(f,f)+\Gamma_4(f,f,f),
\end{eqnarray*}
where
\begin{eqnarray*}
\Gamma_1(f,g)&=&\nu_p Pf=\sum_{i,j} \mathcal{Q}^{\nu}_i\langle f,e_i\rangle\langle g,e_j\rangle e_j,\cr
\Gamma_2(f,g)&=&\nu_p f=\sum_{i,j} \mathcal{Q}^{\nu}_i\langle f,e_j\rangle g,\cr
\Gamma_3(f,g)&=&\nu_c\sum_{i,j} \mathcal{Q}^{\mathcal{M}}_{i,j}\langle f,e_i\rangle\langle g,e_j\rangle,\cr
\Gamma_4(f,g,h)&=&\nu_c\sum_{i,j,k} \mathcal{Q}^{\nu}_i\mathcal{Q}^{\mathcal{M}}_{j,k}\langle f,e_i\rangle\langle g,e_j\rangle
\langle h,e_k\rangle.
\end{eqnarray*}
%
%
%
%
First we note that the conservation laws (\ref{conservation.law.f}) now take the following form:
\begin{lemma} Suppose $f$ is a smooth solution of (\ref{Linearized_Main}). Then $f$ satisfies the following conservation laws.
\begin{align}
\begin{aligned}\label{conservation.law.f}
&\int_{\bbt^3\times\bbr^3}f\sqrt{m}dxdv=0,\cr
&\int_{\bbt^3\times\bbr^3}fv\sqrt{m}dxdv=0,\cr
&\int_{\bbt^3\times\bbr^3}f|v|^2\sqrt{m}dxdv=0.
\end{aligned}
\end{align}
\end{lemma}
We close this section by recalling the following important properties of the linearized Boltzmann-BGK model.
\begin{lemma} The macroscopic projection
\[
Pf\equiv\sum_{i=1}^{5}\langle f,e_i\rangle e_i
\]
is a compact operator from $L^2$ into $L^2$.
\end{lemma}
\begin{proof}
It follows directly from the fact that the kernel of each integral operator lies in $L^2(\bbt^3\times\bbr^3)$
\end{proof}
\begin{lemma}\label{coercivity1}$L$ satisfies the following coercivity property.
\begin{eqnarray*}
\langle Lf,f\rangle= -\nu_c\|(I-P)f\|^2_{L^2_{x,v}}
\end{eqnarray*}
\end{lemma}
\begin{proof}
By $\langle Pf,(I-P)f\rangle=0$, we have
\[
\langle Pf,f\rangle=\langle Pf,Pf\rangle=\|Pf\|^2_{L^2_{x,v}},
\]
which yields
\begin{eqnarray*}
\langle Lf,f\rangle&=&\nu_c\langle Pf,f\rangle-\nu_c\|f\|^2_{L^2_{x,v}}\cr
&=&-\nu_c\|(I-P)f\|^2_{L^2_{x,v}}.
\end{eqnarray*}
\end{proof}
\section{Main result}
%
%
%
%
%
We are finally in a position to state our main result.
\begin{theorem}
Let $N>4$ and $F_0=m+\sqrt{m}f_0\geq0$. Suppose that $f_0$ satisfies the conservation laws (\ref{conservation.law.f}). Then
there exist positive constants $C$, $M$, $\delta^*$ and $\delta_*$ such that if $E(0)<M$,
then there exists a unique global solution $f(x,v,t)$ to (\ref{Linearized_Main}) such that\newline
\indent$(1)$ The high order energy norm is uniformly bounded:
\[
E(t)\leq CE(0).
\]
\indent$(2)$ The perturbation decays exponentially fast:
\[
|||f(t)|||(t)\leq e^{-\delta^*t}\sqrt{E(0)}.
\]
\indent$(3)$ if $\bar f$ denotes another solution corresponding to initial data $\bar f_0$ satisfying the same
assumptions, then we have the following uniform $L^2$-stability estimate:
\[
\|f(t)-\bar f(t)\|_{L^2_{x,v}}\leq e^{-\delta_*t}\|f_0-\bar f_0\|_{L^2_{x,v}}.
\]
\end{theorem}
\begin{remark} Extension of these results to collision frequencies of the following form is straightforward.
\[
\nu_{\eta,\mu}(\rho,T)=\Big\{\sum^{m_{1}}_{i=0}a_i\rho^{\eta_i}(x,t)\Big\}\Big\{\sum^{m_2}_{j}b_jT^{\mu_j}(x,t)\Big\}.
\]
\end{remark}
\section{Preliminary estimates}
In this section, we present several estimates on macrosopic fields which are crucial to develop the argument further.
%
%
%
\begin{lemma}\label{ULBoundofField1}
Let $|\alpha|\geq 1$. Suppose $E(t)$ is sufficiently small. Then we have the following upper and lower bounds for
macroscopic fields:
\begin{eqnarray*}
&&(1)~1-\sqrt{E(t)}\leq\rho(x,t)\leq 1+\sqrt{E(t)},\cr
&&(2)~|U(x,t)|\leq 3\sqrt{E(t)},\cr
&&(3)~\frac{1}{2}\leq T(x,t)\leq \frac{3}{2},\cr
&&(4)~|\partial^{\alpha}\rho(x,t)|\leq \sqrt{E(t)},\cr
&&(5)~|\partial^{\alpha}U(x,t)|\leq C_{|\alpha|}E(t),\cr
&&(6)~|\partial^{\alpha}T(x,t)|\leq C_{|\alpha|}E(t),
\end{eqnarray*}
for some positive constant $C_{|\alpha|}$.
\end{lemma}
\begin{proof}
$(1)$ We have from H\"{o}lder inequality
\begin{eqnarray*}
\rho
=1+\int f\sqrt{m}dv
\leq1+\|f\|_2
\leq1+\sqrt{E(t)}.
\end{eqnarray*}
Similarly, we have
\begin{eqnarray*}
\rho&\geq&1-\int f\sqrt{m}dv\geq1-\|f\|_2\geq 1-\sqrt{E(t)}.
\end{eqnarray*}
$(2)$ Since $\int m vdv=0$, we have by H\"{o}lder inequality
\begin{eqnarray*}
U&=&\frac{\int (m+\sqrt{m}f)vdv}{\rho}=\frac{\int fv\sqrt{m}dv}{\rho}\cr
&\leq&\frac{\frac{3}{2}\|f\|_2}{1-\sqrt{E(t)}}\leq\frac{3}{2}\frac{\sqrt{E(t)}}{1-\sqrt{E(t)}}\cr
&\leq&3\sqrt{E(t)}.
\end{eqnarray*}
$(3)$ The estimate of $T$ can be treated similarly as follows:
\begin{eqnarray*}
T&=&\frac{\int (m+\sqrt{m}f)|v|^2dv-\rho|U|^2}{3\rho}\cr
&\leq&\frac{3+\int f|v|^2\sqrt{m}dv}{3\rho}\cr
&\leq&\frac{3+240\sqrt{2\pi^3}\|f\|_2}{3\rho}\cr
&\leq&\frac{1+80\sqrt{2\pi^3}\sqrt{E(t)}}{(1-\sqrt{E(t)})}\cr
&\leq& \frac{3}{2},
\end{eqnarray*}
where we used the smallness assumption on $E(t)$ and
\[
\int |v|^4e^{-\frac{|v|^2}{2}}dv\leq 240\sqrt{2\pi^3}.
\]
The lower bound can be estimated analogously as follows:
\begin{eqnarray*}
T&=&\frac{\int (m+\sqrt{m}f)|v|^2dv-\rho|U|^2}{3\rho}\cr
&\geq&\frac{3-\int f|v|^2\sqrt{m}dv-\rho|U|^2}{3\rho}\cr
&\geq&\frac{3-15\sqrt{\pi}\sqrt{E(t)}-(1-\sqrt{E(t)})|3\sqrt{E(t)}|^2}{3(1-\sqrt{E(t)})}\cr
&\geq&\frac{1-(5\sqrt{\pi}+3)\sqrt{E(t)}}{1-\sqrt{E(t)}}\cr
&\geq& \frac{1}{2}~.
\end{eqnarray*}
We now turn to the derivatives of the macroscopic fields.\newline
$(4)$ follows directly by the same argument as in (1) noting that
\begin{eqnarray*}
\partial^{\alpha}\rho=\partial^{\alpha}\Big(\int m+f\sqrt{m}dv\Big)=\int\partial^{\alpha}f\sqrt{m}dv.
\end{eqnarray*}
$(5)$ We observe from $U=\frac{\int f\sqrt{m}dv}{\rho}$ that
\begin{eqnarray*}
\displaystyle|\partial^{\alpha}U|\leq C_{|\alpha|}
\displaystyle\Big\{\sum_{1\leq i\leq |\alpha|}|\rho|^{2i}\Big\}
\displaystyle\Big\{\sum_{1\leq \gamma\leq |\alpha|}|\partial\rho|^{2i}\Big\}
\cdots \displaystyle\Big\{\sum_{1\leq \gamma\leq |\alpha|}|\partial^{\alpha}\rho|^{2i}\Big\}\rho^{-2^{|\alpha|}}.
\end{eqnarray*}
We now employ $(1)$ and $(4)$ to see
\begin{eqnarray*}
|\partial^{\alpha}U|\leq \frac{C_{|\alpha|}E(t)}{(1-\sqrt{E(t)})^{2^{|\alpha|}}},
\end{eqnarray*}
where we used for $i\geq 1$
\[
E^{i}(t)\leq E(t).
\]
$(6)$ Similarly, we observe that
\begin{eqnarray*}
\displaystyle|\partial^{\alpha}T|&\leq& C_{|\alpha|}
\displaystyle\Big\{\sum_{1\leq i\leq |\alpha|}|\rho|^{2i}\Big\}
\displaystyle\Big\{\sum_{1\leq \gamma\leq |\alpha|}|\partial\rho|^{2i}\Big\}
\cdots \displaystyle\Big\{\sum_{1\leq \gamma\leq |\alpha|}|\partial^{\alpha}\rho|^{2i}\Big\}\cr
&\times&\displaystyle\Big\{\sum_{1\leq i\leq |\alpha|}|U|^{2i}\Big\}
\displaystyle\Big\{\sum_{1\leq \gamma\leq |\alpha|}|\partial U|^{2i}\Big\}
\cdots \displaystyle\Big\{\sum_{1\leq \gamma\leq |\alpha|}|\partial^{\alpha}U|^{2i}\Big\}\rho^{-2^{|\alpha|}}.
\end{eqnarray*}
This gives by $(1),(2),(4)$ and $(5)$
\begin{eqnarray*}
|\partial^{\alpha}T|\leq \frac{C_{|\alpha|}E(t)}{(1-\sqrt{E(t)})^{2|\alpha|}}.
\end{eqnarray*}
\end{proof}
The following lemma can be proved in an almost identical manner. We omit the proof.
%
%
%
%
\begin{lemma}\label{ULBoundofField2}
Let $|\alpha|\geq 1$. Suppose $E(t)$ is sufficiently small. Then we have
\begin{eqnarray*}
&&(1)~1-\sqrt{E(t)}\leq\rho_{\theta}(x,t)\leq 1+\sqrt{E(t)},\cr
&&(2)~|U_{\theta}(x,t)|\leq 3\sqrt{E(t)},\cr
&&(3)~\frac{1}{2}\leq T_{\theta}(x,t)\leq \frac{3}{2},\cr
&&(4)~|\partial^{\alpha}\rho_{\theta}(x,t)|\leq \sqrt{E(t)},\cr
&&(5)~|\partial^{\alpha}U_{\theta}(x,t)|\leq C_{|\alpha|}E(t),\cr
&&(6)~|\partial^{\alpha}T_{\theta}(x,t)|\leq C_{|\alpha|}E(t),
\end{eqnarray*}
for some positive constant $C_{|\alpha|}$.
\end{lemma}
Having established the preceding estimates for the macroscopic fields, we can now prove the following
crucial proposition for the nonlinear perturbation $\Gamma(f)$.
%
%
%
%
\begin{proposition}\label{bilinear.estimate}
Suppose $E(t)$ is sufficiently small such that estimates in Lemma \ref{ULBoundofField1} and Lemma \ref{ULBoundofField2}
are valid. Then we have
\begin{eqnarray*}
&&(1)~\Big|\int\partial^{\alpha}_{\beta}\Gamma(f,f,f)rdv\Big|\leq C\sum_{|\alpha_1|+|\alpha_2|\leq |\alpha|}
\|\partial^{\alpha_1}f\|_{L^2_{x,v}}\|\partial^{\alpha_2}f\|_{L^2_{v}}\|r\|_{L^2_{v}}\cr
&&\hspace{4cm}+~C\sum_{\substack{|\alpha_1|+|\alpha_2|\leq |\alpha|,\\|\beta_2|\leq |\beta|}}
\|\partial^{\alpha_1}f\|_{L^2_{x,v}}\|\partial^{\alpha_2}_{\beta_2}f\|_{L^2_{v}}\|r\|_{L^2_{v}}\cr
&&\hspace{4cm}+~C\sum_{\substack{|\alpha_1|+|\alpha_2|+|\alpha_3|\\\leq |\alpha|}}C\|\partial^{\alpha_1}f\|_{L^2_{x,v}}\|\partial^{\alpha_2}f\|_{L^2_{v}}\|\partial^{\alpha_3}f\|_{L^2_{v}}\|r\|_{L^2_{v}},\cr
&&(2)~\Big|\langle\Gamma_{1,2,3}(f,g)f\rangle\Big|+\Big|\langle\Gamma_{1,2,3}(g,f)f\rangle\Big|\leq C\sup_{x}\|g\|_{L^2_{x,v}}\|f\|^2_{L^2_{x,v}},\cr
&&\hspace{0.63cm}\Big|\langle\Gamma_{4}(f,g,h)f\rangle\Big|+\Big|\langle\Gamma_4(g,f,h)f\rangle\Big|+\Big|\langle\Gamma_4(g,h,f)f\rangle\Big|
\leq C\sup_{x}\|g\|_{L^2_{v}}\sup_{x}\|h\|_{L^2_{v}}\|f\|^2_{L^2_{x,v}},\cr
&&(3)~\Big\|\Gamma_{1,2,3}(f,g)r+\Gamma_{1,2,3}(g,f)r\Big\|_{L^2_{x,v}}\leq C\sup_{x,v}| r|\sup_x\|f\|_{L^2_{v}}\|g\|_{L^2_{x,v}},\cr
&&\hspace{0.63cm}\Big\|\Gamma_{4}(f,g,h)r+\Gamma_{4}(g,f,h)r+\Gamma_{4}(g,h,f)r\Big\|_{L^2_{x,v}}\leq C\sup_{x,v}| r|\sup_x\|f\|_{L^2_{v}}\sup_x\|g\|_{L^2_{v}}\|h\|_{L^2_{x,v}}.
\end{eqnarray*}
\end{proposition}
\begin{remark}
Note that, unlike the case of the Boltzmann equation, we need to impose the smallness condition on the high order energy
to prove the estimates.
\end{remark}
\begin{proof}
(1) To prove (1), we should consider $\Gamma_i(f)$  $(1\leq i\leq 4)$ separately. But for simplicity we only present the
proof for $\Gamma_3(f)$.
Other estimates can be obtained in an almost identical manner. \newline
The estimate of $\Gamma_3(f)$: We first prove the following claim:\newline
Claim: ~There exists a positive constant $\varepsilon=\varepsilon(\alpha,\beta)$ such that
\begin{eqnarray*}
\Big|\partial^{\alpha}_{\beta}\Big\{Q^{\mathcal{M}}_{ij}
e^{-\frac{|v-U_{\theta}|^2}{2T_{\theta}}+\frac{|v|^2}{4}}\Big\}\Big|
\leq Ce^{-\frac{|v-U_{\theta}|^2}{(2+\varepsilon)T_{\theta}}+\frac{|v|^2}{4}}.
\end{eqnarray*}
(Proof of the claim):\newline
We apply the differential operator $\partial$ to $Q^{\mathcal{M}}_{ij}e^{-\frac{|v-U_{\theta}|^2}{2T_{\theta}}+\frac{|v|^2}{4}}$
to see
\begin{eqnarray*}
\partial\Big\{Q^{\mathcal{M}}_{ij}e^{-\frac{|v-U_{\theta}|^2}{2T_{\theta}}+\frac{|v|^2}{4}}\Big\}
&=&\partial\Big\{ Q^{\mathcal{M}}_{ij}\Big\}e^{-\frac{|v-U_{\theta}|^2}{2T_{\theta}}+\frac{|v|^2}{4}}
+Q^{\mathcal{M}}_{ij}\partial\Big\{e^{-\frac{|v-U_{\theta}|^2}{2T_{\theta}}+\frac{|v|^2}{4}}\Big\}\cr
&=&Q^{\mathcal{M}}_{ij}e^{-\frac{|v-U_{\theta}|^2}{2T_{\theta}}+\frac{|v|^2}{4}}\cr
&+& Q^{\mathcal{M}}_{i,j}\Big\{\partial T\frac{|v-U_{\theta}|^2}{2T^2_{\theta}}+\partial U\cdot\frac{v-U}{2T}-\frac{v}{2}\Big\}
e^{-\frac{|v-U_{\theta}|^2}{2T_{\theta}}+\frac{|v|^2}{4}}\cr
&=&\frac{P^{\mathcal{M}}_{ij}(\rho_{\theta},U_{\theta},T_{\theta},v-U_{\theta},v)}{R^{\mathcal{M}}_{ij}(\rho_{\theta},U_{\theta},T_{\theta},\mathcal{G}_{\theta})}
e^{-\frac{|v-U_{\theta}|^2}{2T_{\theta}}+\frac{|v|^2}{4}}\cr
&\leq&\frac{P^{\mathcal{M}}_{ij}(\rho_{\theta},U_{\theta},T_{\theta},1,1)}{R^{\mathcal{M}}_{ij}(\rho_{\theta},U_{\theta},T_{\theta},\mathcal{G}_{\theta})}
e^{-\frac{|v-U_{\theta}|^2}{(2+\varepsilon)T_{\theta}}+\frac{|v|^2}{4}}\cr
&\leq&Ce^{-\frac{|v-U_{\theta}|^2}{(2+\varepsilon)T_{\theta}}+\frac{|v|^2}{4}},
\end{eqnarray*}
where we used the upper and lower bounds of Lemma \ref{ULBoundofField2} with
\begin{eqnarray*}
\mathcal{G}_{\theta}&=&3+3\rho_{\theta}-|U_{\theta}|^2-3T_{\theta}\cr
&\geq&3+3(1-\sqrt{E(t)})-(3\sqrt{E(t)})^2-\frac{9}{2}\cr
&=&\frac{3}{2}-3\sqrt{E(t)}-9E(t)\cr
&\geq&1
\end{eqnarray*}
and
\begin{eqnarray*}
&&|v-U_{\theta}|^re^{-\frac{|v-U_{\theta}|^2}{(2+\varepsilon)T_{\theta}}}
<C\{(2+\varepsilon)T_{\theta}\}^{\frac{\gamma}{2}}<\infty,\cr
&&|v|^re^{-\frac{|v-U_{\theta}|^2}{(2+\varepsilon)T_{\theta}}}\leq C_r(|v-U_{\theta}|^r+|U_{\theta}|^r)
e^{-\frac{|v-U_{\theta}|^2}{(2+\varepsilon)T_{\theta}}}<\infty.
\end{eqnarray*}
Then the induction argument gives the desired result.
We now employ the claim and use H\"{o}lder inequality to see
\begin{align}
\begin{aligned}\label{TRI1}
&\int\big|\partial^{\alpha}_{\beta}\Gamma_3(f,f)r\big|dv\cr
&\hspace{1cm}\leq\sum_{\substack{|\alpha_0|+|\alpha_1|+|\alpha_2|\\=|\alpha|}}
\int \Big|\partial^{\alpha_0}_{\beta}\Big\{Q^{\mathcal{M}}_{i,j}e^{-\frac{|v-U_{\theta}|^2}{2T_{\theta}}+\frac{|v|^2}{4}}\Big\}\Big|
\langle\partial^{\alpha_1}f,e_i\rangle\langle\partial^{\alpha_2}f,e_j\rangle rdv\cr
&\hspace{1cm}\leq C\sum_{\substack{|\alpha_0|+|\alpha_1|+|\alpha_2|\\=|\alpha|}}
\int e^{-\frac{|v-U_{\theta}|^2}{(2+\varepsilon)T_{\theta}}+\frac{|v|^2}{4}}
\langle\partial^{\alpha_1}f,e_i\rangle\langle\partial^{\alpha_2}f,e_j\rangle rdv\cr
&\hspace{1cm}=C\sum_{\substack{|\alpha_0|+|\alpha_1|+|\alpha_2|\\=|\alpha|}}
\|\partial^{\alpha_1}f\|_{L^2_{v}}\|\partial^{\alpha_2}f\|_{L^2_{v}}
\int e^{-\frac{|v-U_{\theta}|^2}{(2+\varepsilon)T_{\theta}}+\frac{|v|^2}{4}}rdv.
\end{aligned}
\end{align}
We apply H\"{o}lder inequality again to see
\begin{eqnarray}\label{TRI2}
\int e^{-\frac{|v-U_{\theta}|}{(2+\varepsilon)T_{\theta}}+\frac{|v|^2}{4}}hdv
&\leq&C\Big\|e^{-\frac{|v-U_{\theta}|}{(2+\varepsilon)T_{\theta}}+\frac{|v|^2}{4}}\Big\|_{L^2_{v}}\|r\|_{L^2_{v}}\leq C\|r\|_{L^2_{v}},
\end{eqnarray}
where we used
\begin{eqnarray*}
\Big\|e^{-\frac{|v-U_{\theta}|}{2T_{\theta}}+\frac{|v|^2}{4}}\Big\|^2_2
&=&\int \exp\Big(-\Big[\frac{2}{(2+\varepsilon)T_{\theta}}-\frac{1}{2}\Big]
\Big|v-\frac{4}{4-(2+\varepsilon)T_{\theta}}U_{\theta}\Big|^2
-\frac{(2+\varepsilon)T_{\theta}}{2-(2+\varepsilon)T_{\theta}}\Big)dv\cr
&\leq&C\int \exp\Big(-\Big[\frac{1}{(2+\varepsilon)T_{\theta}}-\frac{1}{2}\Big]\Big|v+\frac{2}{4-(2+\varepsilon)T_{\theta}}
U_{\theta}\Big|^2\Big)dv\cr
&=&C\sqrt{\Big(\frac{4-(2+\varepsilon)T_{\theta}}{{2}}\Big)^3}<\infty.
\end{eqnarray*}
In the last line, we used Lemma \ref{ULBoundofField2}.
We now substitute (\ref{TRI2}) into (\ref{TRI1}) to obtain
the desired result.\newline
(2) In (1), we set $\alpha$=$\beta$=0.
Then we have from (\ref{TRI1})
\begin{eqnarray*}
\int\Gamma_3(f,g)f dx dv &\leq& C\int\|f\|_{L^2_{v}}\|g\|_{L^2_{v}}
\Big(\int e^{-\frac{|v-U_{\theta}|}{(2+\varepsilon)T_{\theta}}+\frac{|v|^2}{4}}fdv\Big)dx\cr
&\leq&C\int\|f\|_{L^2_{v}}\|g\|_{L^2_{v}}\|f\|_{L^2_v}dx \cr
&\leq& C\sup_{x}\|g\|_{L^2_v}\|f\|^2_{L^2_{x,v}}.
\end{eqnarray*}
We now take $L^2$ norms with respect to spatial variables to obtain the desired result.\newline
(3) Let $\phi\in L^2$. Then we have from the same argument used in (1)
\begin{eqnarray*}
\langle \Gamma_3(f,g)r,\phi \rangle &\leq& C\int\|f\|_{L^2_{v}}\|g\|_{L^2_{v}}\|r\phi\|_{L^2_{v}}dx\cr
&\leq&C\sup_{x,v}|r|\int\|f\|_{L^2_{v}}\|g\|_{L^2_{v}}\|\phi\|_{L^2_{v}}dx\cr
&\leq&C\sup_{x,v}|r|\Big(\sqrt{\int\|f\|^2_{L^2_{v}}\|g\|^2_{L^2_{v}}dx}\Big)\|\phi\|_{L^2_{x,v}}.
\end{eqnarray*}
Therefore, the duality argument gives
\begin{eqnarray*}
\big\|\Gamma(f,g)r\big\|&\leq& C\sup_{x,v}|r|\sqrt{\int\|f\|^2_{L^2_{v}}\|g\|^2_{L^2_{v}}dx}\cr
&\leq&C\sup_{x,v}|r|\sup_x\|f\|_{L^2_{v}}\|g\|_{L^2_{x,v}}.
\end{eqnarray*}
\end{proof}
%
%
%
%
%
\section{Local existence}
In this section, we establish the local in time existence of classical solutions under the assumption that
the high order energy $E(t)$ is sufficiently small. This local solution will be extended to the global solution
in the last section by combining the coercivity estimate of $L$ and a refined energy estimate.
\begin{theorem}\label{local.existence}
Let  $F_0=m+\sqrt{m}f_0\geq 0$. Suppose $f_0$ satisfies the conservation laws (\ref{conservation.law.f}).
Then there exist $M_0>0$ , $T_{*}>0$, such that if ~$T^*\leq \frac{M_0}{2}$ and
$E(0)\leq \frac{M_0}{2}$, there is a unique solution $f(x,v,t)$ to the
Boltzmann-BGK (\ref{Linearized_Main}) such that\newline
\indent$(1)$ The high order energy $E(t)$ is continuous in $[0,T^*)$ and uniformly bounded:
\[
\sup_{0\leq t\leq T^*}E(t)\leq M_0.
\]
\indent$(2)$ The distribution function remains positive in $[0, T_{*})$:
\[
F(x,v,t)=m+\sqrt{m}f(x,v,t)\geq 0.
\]
\indent$(3)$ The conservation laws (\ref{conservation.law.f}) hold for all $[0, T_*]$.
\end{theorem}
\begin{proof}
We consider the following iteration sequence:
\begin{align}
\begin{aligned}\label{iteration.F}
\displaystyle\partial_t F^{n+1}  + v \cdot \nabla_x F^{n+1}&=\nu^n(\mathcal{M}(F^n)-F^{n+1}),\cr
\displaystyle F^{n+1}(x,v,0) &= F_0(x,v),
\end{aligned}
\end{align}
which is equivalent to
\begin{align}
\begin{aligned}\label{Iteration_f}
\displaystyle\big\{\partial_t +v\cdot\nabla_x+\nu_c\big\}f^{n+1}&=\nu_cP f^n+\Gamma(f^{n}),\cr
\displaystyle f^{n+1}(x,v,0) &= f_0(x,v). 
\end{aligned}
\end{align}
Now the theorem follows easily once we establish the following lemma.
%
%
%
%
%
\begin{lemma}
There exist $M_0>0$ and $T_*>0$ such that if $E(f_0)<\frac{M_0}{2}$ then
$E(f^n(t))<M_0$ implies $E(f^{n+1}(t))<M_0$ for $t\in [0,T_*]$.
\end{lemma}
\begin{proof}
We take $\partial^{\alpha}_{\beta}$ derivatives of (\ref{Iteration_f}) to obtain
\begin{align}
\begin{aligned}\label{Iteration_Df}
\underbrace{\big\{\partial_t +v\cdot\nabla_x+\nu_c\big\}
\partial^{\alpha}_{\beta}f^{n+1}}_{L}
=&-\underbrace{\sum_{\beta\neq 0}\{\partial_{\beta}v\cdot\nabla_x\}\partial^{\alpha}f^{n+1}}_{R_1}
&+\underbrace{\nu_c\partial_{\beta}P\partial^{\alpha}f^{n}}_{R_2}
+\underbrace{\partial^{\alpha}_{\beta}\Gamma(f^n)}_{R_3}.
\end{aligned}
\end{align}
We take the inner product of (\ref{Iteration_Df}) with $\partial^{\alpha}_{\beta}f^{n+1}$ and estimate each term
separately.\newline
(1) $L$: l.h.s can be calculated directly as follows:
\[
\langle L,\partial^{\alpha}_{\beta}f^{n+1}\rangle=\frac{1}{2}\frac{d}{dt}\|\partial^{\alpha}_{\beta}f^{n+1}(t)\|^2_{L^2_{x,v}}
+\nu_c\|\partial^{\alpha}_{\beta}f(t)\|^2_{L^2_{x,v}}.
\]
(2) $R_1$:
We see from the following observation
\begin{eqnarray*}
\partial_{v_i}v\cdot\nabla_x\partial^{\alpha}f^{n+1}=\partial_{x_i}\partial^{\alpha}f^{n+1}
\end{eqnarray*}
that
\begin{eqnarray*}
\langle R_1,\partial^{\alpha}_{\beta}f\rangle&\leq&
\sum_{\beta\neq 0} \|\{\partial_{\beta}v\cdot\nabla_x\}\partial^{\alpha}f^{n+1}\|_{L^2_{x,v}}
\|\partial^{\alpha}_{\beta}f^{n+1}\|_{L^2_{x,v}}\cr
&\leq&C|||f^{n+1}(t)|||^2\cr
&\leq&CE_{n+1}(t).
\end{eqnarray*}
(3) $R_2$: We first note that
\[
\|\partial_{\beta}P\partial^{\alpha}f\|_{L^2_{x,v}}\leq C_{\beta}\|\partial^{\alpha}f\|_{L^2_{x,v}}.
\]
Therefore, we have from H\"{o}lder inequality and Young's inequality
\begin{eqnarray*}
\langle R_2,\partial^{\alpha}_{\beta}f\rangle\leq &\leq&
C\|\partial^{\alpha}f^{n}\|^2_{L^2_{x,v}}+\|\partial^{\alpha}_{\beta}f^{n+1}\|^2_{L^2_{x,v}}\cr
&\leq& C(|||f^{n}|||^2+|||f^{n+1}|||^2)\cr
&\leq&C(E_{n}(t)+E_{n+1}(t)).
\end{eqnarray*}
We now turn to the estimate of the nonlinear term.\newline
(4) $R_3$: We have from Lemma \ref{bilinear.estimate}
\begin{eqnarray*}
\langle R_4,\partial^{\alpha}_{\beta}f\rangle
&\leq& C\sum_{|\alpha_1|+|\alpha_2|\leq |\alpha|}
\int_{\bbr^3}\|\partial^{\alpha_1}f^n\|_{L^2_{x,v}}\|\partial^{\alpha_2}f^n\|_{L^2_{v}}\|\partial^{\alpha}_{\beta}f^{n+1}\|_{L^2_{v}}dx\cr
&+&C\sum_{\substack{|\alpha_1|+|\alpha_2|\leq |\alpha|,\\|\beta_2|\leq |\beta|}}
\int_{\bbr^3}\|\partial^{\alpha_1}f\|_{L^2_{x,v}}\|\partial^{\alpha_2}_{\beta_2}f^n\|_{L^2_{v}}\|\partial^{\alpha}_{\beta}f^{n+1}\|_{L^2_{v}}dx\cr
&+&C\sum_{\substack{|\alpha_1|+|\alpha_2|+|\alpha_3|\\\leq |\alpha|}}\int_{\bbr^3}\|\partial^{\alpha_1}f\|_{L^2_{x,v}}\|\partial^{\alpha_2}f^n\|_{L^2_{v}}\|\partial^{\alpha_3}f^n\|_{L^2_{v}}\|\partial^{\alpha}_{\beta}f^{n+1}\|_{L^2_{v}}dx\cr
&\leq& C\sum_{|\alpha_1|+|\alpha_2|\leq|\alpha|}
\big(\sup_{x}\|\partial^{\alpha_1}f^{n}\|_{L^2_{v}}+\sup_{x}\|\partial^{\alpha_1}f^{n}\|^2_{L^2_{v}}\big)\int_{\bbr^3}
\|\partial^{\alpha_2}f^{n}\|_{L^2_{v}}\|\partial^{\alpha}_{\beta}f^{n+1}\|_{L^2_{v}}dx\cr
&\leq& C\sum_{|\alpha_1|+|\alpha_2|\leq|\alpha|}\big(\sup_{x}\|\partial^{\alpha_1}f^{n}\|_{L^2_{v}}+
\sup_{x}\|\partial^{\alpha_1}f^{n}\|^2_{L^2_{v}}\big)
\Big(\|\partial^{\alpha_2}f^{n}\|^2_{L^2_{x,v}}
+\|\partial^{\alpha}_{\beta}f^{n+1}\|^2_{L^2_{x,v}}\Big)\cr
&\leq&C\big(E^{\frac{3}{2}}_n(t)+E^2_n(t)+\sqrt{E_n(t)}E_{n+1}(t)+E_n(t)E_{n+1}(t)\big),
\end{eqnarray*}
where we assumed $\alpha_1$ to be the smallest index without loss of generality and used the following Sobolev embedding
\begin{eqnarray}\label{Sobolev}
H^2(\bbt^3)\subseteq L^{\infty}(\bbt^3).
\end{eqnarray}
We substitute all these ingredients into (\ref{Iteration_Df}) to obtain
\begin{eqnarray*}
&&\frac{1}{2}\frac{d}{dt}\|\partial^{\alpha}_{\beta}f^{n+1}(t)\|^2_{L^2_{x,v}}
+\nu_c\|\partial^{\alpha}_{\beta}f^{n+1}(t)\|^2_{L^2_{x,v}}\cr
&&\hspace{1cm}\leq
C\Big\{E_{n+1}(t)+ \big(E_n(t)\big)^{\frac{3}{2}}+\big(E_n(t)\big)^{2}+\big(E_n(t)\big)^{\frac{1}{2}}E_{n+1}(t)+E_n(t)E_{n+1}(t)
\Big\}.
\end{eqnarray*}
We then sum over $\alpha$ and $\beta$ and integrate in time to see
\begin{align}
\begin{aligned}\nonumber
&E_{n+1}(t)\leq E_{n+1}(0)\cr
&\hspace{0.5cm}+C\int^t_0\Big\{E_{n+1}(t)+\big(E_n(t)\big)^{\frac{3}{2}}+\big(E_n(t)\big)^{2}
+\big(E_n(t)\big)^{\frac{1}{2}}E_{n+1}(t)+E_n(t)E_{n+1}(t)\Big\}ds\cr
&\hspace{0.5cm}\leq\frac{M_0}{2}+C\Big\{T_*\sup_{0\leq t\leq T_*}E_{n+1}(t)
+T_*\sup_{0\leq t\leq T_*}E_{n}
+T_*\big(\sup_{0\leq t\leq T_*}E_{n}\big)^{\frac{3}{2}}+T_*\big(\sup_{0\leq t\leq T_*}E_{n}\big)^2\cr
&\hspace{0.5cm}+T_*\Big(\sup_{0\leq t\leq T_*}E_n\Big)^{\frac{1}{2}}\sup_{0\leq T_*}E^{n+1}(t)
+T_*\Big(\sup_{0\leq t\leq T_*}E_n\Big)\sup_{0\leq T_*}E^{n+1}(t)
\Big\},
\end{aligned}
\end{align}
which yields
\begin{eqnarray*}
\Big(1-CT_*-CT_*\sqrt{M_0}-CT_*M_0~\Big)\sup_{0\leq t\leq T_*}E^{n+1}(t)\leq
\Big(\frac{1}{2}+CT_*+CT_*\sqrt{M_0}+CT_*M_0~\Big)M_0.
\end{eqnarray*}
This gives the desired result for sufficiently small $M_0$ and $T_*$.\newline
\end{proof}
We now go back to the proof of the theorem and let $n\rightarrow \infty$ to establish the local in time existence
of a smooth solution.
To prove uniqueness, we assume that $g$ is another local solution corresponding to the same initial data $f_0$.
We then have
\begin{align}
\begin{aligned}\label{Uniqueness1}
\big\{\partial_t+v\cdot \nabla +\nu_c\big\}(f-g)&=P(f-g)+\Gamma_{1,2,3}(f-g,f)+\Gamma_{1,2,3}(g,f-g)\cr
&+\Gamma_4(f-g,f,f)+\Gamma_4(g,f-g,f)+\Gamma_4(g,g,f-g).
\end{aligned}
\end{align}
We recall from Proposition \ref{bilinear.estimate} and (\ref{Sobolev})
\begin{eqnarray*}
&&\langle\Gamma_{1,2,3}(f-g,f)+\Gamma_{1,2,3}(g,f-g),f-g\rangle\cr
&&\hspace{1cm}+\langle\Gamma_4(f-g,f,f)+\Gamma_4(g,f-g,f)+\Gamma_4(g,g,f-g),f-g\rangle\cr
&&\hspace{1cm}\leq\sum_{|\alpha|\leq 2}\big(\|\partial^{\alpha}f\|^2_{L^2_{x,v}}
+\|\partial^{\alpha}g\|^2_{L^2_{x,v}}+\|\partial^{\alpha}f\|_{L^2_{x,v}}
+\|\partial^{\alpha}g\|_{L^2_{x,v}}\big)\|f-g\|^2_{L^2_{x,v}}.
\end{eqnarray*}
We now multiply $f-g$ to both sides of (\ref{Uniqueness1}), integrate with respect to $x$, $v$, $t$ and use the above estimate to see
\begin{eqnarray*}
&&\|f(t)-g(t)\|^2_{L^2_{x,v}}+\int^t_0\|f(s)-g(s)\|^2_{L^2_{x,v}}ds\cr
&&\hspace{0.5cm}\leq
C\sum_{|\alpha|\leq 2}\sup_{0\leq t\leq T_*}\big(\sqrt{E_f(t)}+\sqrt{E_g(t)}+E_f(t)+E_g(t)+1\big)
\int^t_0\|f(s)-g(s)\|^2_{L^2_{x,v}}ds.
\end{eqnarray*}
Therefore, for sufficiently small $E_f(0)$ and $E_g(0)$, the uniqueness follows from Grownwall's theorem.
We now turn to the continuity of $E(t)$. Let $f$ be the smooth local solution constructed above:
\begin{eqnarray*}
\partial_t f+v\cdot \nabla f+\nu_cf=Pf+\Gamma(f).
\end{eqnarray*}
We multiply $f$ and integrate over $x$, $v$ and then over $[s,t]$ to see.
\begin{eqnarray*}
|E(t)-E(s)|
\leq C\Big[1+\sqrt{E(t)}+E(t)\Big]\int^t_s\sum_{|\alpha|\leq 2}\|\partial^{\alpha}f\|_{L^2_{x,v}}d\tau\rightarrow 0.
\end{eqnarray*}
The positivity of $m+\sqrt{m}f$ can be verified iteratively from the positivity of $F_0$ using (\ref{iteration.F}).
Finally, since the local solution is smooth, the conservation laws can be obtained straightforwardly.
\end{proof}
\section{Coercivity of L}
The coercivity estimate in Lemma \ref{coercivity1}, involving only microscopic components,
is not strong enough to play as a good term in the energy method. In this section, we show that the full coercivity
$L$ can be recovered as long as the energy $E(t)$ remains sufficiently small.
We first set for simplicity
\[a(x,t)=\int f\sqrt{m}dv,~ b(x,t)=\int fv\sqrt{m}dv \mbox{ and }c(x,t)=\int f|v|^2\sqrt{m}dv.\]
Recall that $f$ can be divided into its hydrodynamic part $\tilde{P}f$ and microscopic part $(I-\tilde P)f$:
\begin{eqnarray}\label{divide}
f=\tilde{P}f+(I-\tilde{P})f,
\end{eqnarray}
where
\[
\tilde{P}f=a\sqrt{m}+b\cdot v\sqrt{m}+c|v|^2\sqrt{m}.
\]
We observe that there exists constants $C$ such that
\begin{eqnarray}\label{eqivalence}
\frac{1}{C}\|(I-P)f\|_{L^2_{x,v}}\leq\|(I-\tilde{P})f\|_{L^2_{x,v}}\leq C\|(I-P)f\|_{L^2_{x,v}}.
\end{eqnarray}
Now, we substitute (\ref{divide}) into the BGK model (\ref{Linearized_Main}) to obtain
\begin{eqnarray}\label{(Pf)+(l-P)f}
\{\partial_t+v\cdot\nabla\}\tilde{P}f=\ell\big\{(I-\tilde{P})f\big\}+h(f),
\end{eqnarray}
where
\begin{eqnarray*}
\ell\big\{(I-\widetilde{P})f\big\}&\equiv&\{-\partial_t-v\cdot\nabla_x+L\}\{I-\tilde{P}\}f,\cr
h(f)&\equiv&\Gamma(f).
\end{eqnarray*}
The l.h.s of (\ref{(Pf)+(l-P)f}) is calculated as follows
\begin{eqnarray*}
\sum_{i}\Big\{v_i\partial^{i}c|v|^2+(\partial_tc+\partial^ib_i)v_i^2
+\sum_{j>i}(\partial^ib_j+\partial^jb_i)v_iv_j+(\partial_tb+\partial^ia)v_i+\partial_ta\Big\}\sqrt{m},
\end{eqnarray*}
where $\partial^{i}=\partial_{x_i}$. We then expand the l.h.s of (\ref{(Pf)+(l-P)f}) with respect to the basis:
\[\sqrt{m},~ v_i\sqrt{m},~ v_iv_j\sqrt{m},~ v_i^2\sqrt{m},~ v_i|v|^2\sqrt{m}\quad(1\leq i,j\leq 3).\]
Equating both sides of the above identity, we obtain the following result.
\begin{lemma}\label{MicroMacroEquations} $a$, $b$, $c$ satisfy the following relations.
\begin{eqnarray*}
\begin{array}{ll}
(1)&\nabla c=\ell_{c}+h_{c}\label{MM1},\\
(2)&\partial_t c+\partial^i b_i=\ell_{i}+h_{i}\label{MM2},\\
(3)&\partial^i b_j+\partial^j b_i =\ell_{ij}+h_{ij}\label{MM3},\\
(4)&\partial_t b_i+\partial^i a=\ell_{bi}+h_{bi}\label{MM4},\\
(5)&\partial_t a=\ell_{a}+h_{a}\label{MM5},
\end{array}
\end{eqnarray*}
where $\ell_{c}$, $\ell_{i}$, $\ell_{ij}$, $\ell_{bi}$, $\ell_{a}$ are coefficients of
the expansion of $\ell$ with respect to the preceding basis. Similarly, $h_{c}$, $h_{i}$, $h_{ij}$, $h_{bi}$, $h_{a}$ denotes
the corresponding coefficients of the expansion of $h$.
\end{lemma}
For brevity, we define $\tilde\ell$ and $\tilde h$ as
\begin{eqnarray*}
\tilde\ell&\equiv&\ell_c+\sum_i\ell_i+\sum_{i,j}\ell_{ij}+\sum_{i}\ell_{bi}+\ell_a,\cr
\tilde h&\equiv&h_c+\sum_ih_i+\sum_{i,j}h_{ij}+\sum_{i}h_{bi}+h_a.
\end{eqnarray*}
%
%
%
\begin{lemma}\label{EllipticEstimate_b} Let $|\alpha|\leq N-1$. Then we have
\begin{eqnarray*}
&&(1)~\|\nabla_x \partial^{\alpha}b_i\|_{L^2_x}+\|\partial_i\partial^{\alpha}b_i\|_{L^2_x}\leq
\sum_{|\alpha|\leq N-1}\|\partial^{\alpha}\tilde\ell\|_{L^2_{x,v}}
+\sum_{|\alpha|\leq N-1}\|\partial^{\alpha}\tilde h\|_{L^2_{x,v}},\cr
&&(2)~\|\partial^{\alpha}_tb_i\|_{L^2_x}\leq \sum_{|\alpha|\leq N-1}\|\partial^{\alpha}\tilde\ell\|_{L^2_{x,v}}
+\sum_{|\alpha|\leq N-1}\|\partial^{\alpha}\tilde h\|_{L^2_{x,v}}.
\end{eqnarray*}
\end{lemma}
\begin{proof}
$(1)$ Following \cite{Guo1,Guo2,Guo3}, we observe that
\begin{eqnarray*}
&&-\Delta b_i-\partial_i\partial_i b_i\cr
&&\hspace{1cm}=-\sum_{j\neq i}\partial^{j}(\partial^{j}b_i)-2\partial^{i}\partial^{i}b_i\cr
&&\hspace{1cm}=-\sum_{j\neq i}\underline{\partial^j(-\partial^i b_j}-\ell_{ij}-h_{ij})
-\underline{2\partial^{i}(-\partial_t c}+\ell_i+h_i)\cr
&&\hspace{2cm}(\mbox{ by Lemma }\ref{MicroMacroEquations} ~(3)\mbox{ and } (2))\cr
&&\hspace{1cm}=\sum_{j\neq i}(\partial^j\partial^i b_j+\partial^i\partial_t c)
+\sum_{j\neq i}(\partial^j\ell_{ij}-\partial^jh_{ij})-2\partial^{i}(\ell_i+h_i)\cr
&&\hspace{1cm}=\sum_{j\neq i}(\partial_i(-\partial_tc-\ell_{i}-h_i)+\partial^i\partial_t c)
+\sum_{j\neq i}(\partial^j\ell_{ij}-\partial^jh_{ij})-2\partial^{i}(\ell_i+h_i)\cr
&&\hspace{2cm}(\mbox{ by Lemma }\ref{MicroMacroEquations}~ (2) )\cr
&&\hspace{1cm}=\sum_{j\neq i}(\partial_i\ell_{i}+\partial^i h_i)
+\sum_{j\neq i}(\partial^j\ell_{ij}-\partial^jh_{ij})-2\partial^{i}(\ell_i+h_i).
\end{eqnarray*}
Then the result follows from the standard elliptic estimate.\newline
$(2)$ By Poincare inequality and Lemma \ref{EllipticEstimate_b}, we have
\begin{eqnarray*}
\|\partial^{\gamma_t}_tb\|_{L^2_x}&\leq&\|\nabla_x\partial^{\gamma_t} b\|_{L^2_x}\cr
&\leq&C\big(\|\partial^{\gamma_t}\tilde\ell\|_{L^2_x}+\|\partial^{\gamma_t}\tilde h\|_{L^2_x}\big).
\end{eqnarray*}
\end{proof}
%
%
%
\begin{lemma}\label{Estimate_C}
For $|\alpha|\leq N-1$, we have
\begin{eqnarray*}
&&(1)~ \|c\|_{L^2_x}\leq C\big(~\|\tilde\ell\|_{L^2_x}+\|\tilde h\|_{L^2_x}~\big).\cr
&&(2)~ \|\partial_t\partial^{\alpha} c\|_{L^2_x}\leq C\big(~\|\partial^{\alpha}\tilde\ell\|_{L^2_x}
+\|\partial^{\alpha}\tilde h\|_{L^2_x}~\big).\cr
&&(3)~ \|\nabla_x \partial^{\alpha}c\|_{L^2_x}\leq C\big(~\|\partial^{\alpha}\tilde\ell\|_{L^2_x}
+\|\partial^{\alpha}\tilde h\|_{L^2_x}~\big).
\end{eqnarray*}
\end{lemma}
\begin{proof}
(1) By Poincare inequality and Lemma \ref{MicroMacroEquations}(1), we have
\begin{eqnarray*}
\|c\|_{L^2_x}&\leq&\|\nabla c\|_{L^2_x}\cr
&\leq&C\|\ell_c+h_c\|_{L^2_x}\cr
&\leq&C\big(\|\tilde\ell\|_{L^2_x}+\|\tilde h\|_{L^2_x}\big),
\end{eqnarray*}
where we used the conservation of energy:
\[\int c(x)dx=0.\]
(2) By (2) in Lemma \ref{MicroMacroEquations}, we have
\begin{eqnarray*}
\|\partial_tc\|_{L^2_x}&=&\|-\nabla\cdot b+\ell+h\|_{L^2_x}\cr
&\leq& C\big(~\|\tilde \ell \|_{L^2_x}+\|\tilde h\|_{L^2_x}).
\end{eqnarray*}
(3) follows directly from (1).
\end{proof}
%
%
%
\begin{lemma}\label{Estimate_A1}
Let $|\alpha|\leq N-1$. For (2), we assume further that $\alpha$ is purely spatial: $\alpha=[0,\alpha_1,\alpha_2,\alpha_3]\neq0$.
Then we have
\begin{eqnarray*}
&&(1)~\|\partial_t \partial^{\alpha}a\|_{L^2_x}\leq C(~\|\partial^{\alpha}\tilde\ell\|_{L^2_x}+\|\partial^{\alpha}\tilde h\|_{L^2_x}~),\cr
&&(2)~\|\nabla \partial^{\alpha}_xa\|_{L^2_x}\leq \sum_{|\bar\alpha|\leq N-1}\|\partial^{\bar\alpha}\tilde\ell\|_{L^2_x}
+\sum_{|\bar\alpha|\leq N-1}\|\partial^{\bar\alpha}\tilde h\|_{L^2_x},\cr
&&(3)~\|a\|_{L^2_x}\leq \|\tilde\ell\|_{L^2_x}+\|\tilde h\|_{L^2_x}.
\end{eqnarray*}
\end{lemma}
\begin{proof}
(1) This follows directly from Lemma \ref{MicroMacroEquations} (5).\newline
(2) By Lemma \ref{MicroMacroEquations} (4), we have
\begin{eqnarray*}
\triangle \partial^{\alpha}a&=&\nabla\cdot\nabla \partial^{\alpha}a\nonumber\cr
&=&\nabla\cdot(-\partial^t \partial^{\alpha}b+\partial^{\alpha}\ell_b+\partial^{\alpha}h_b)\nonumber\cr
&=&-\nabla\partial^t \partial^{\alpha}b+\nabla\partial^{\alpha}\ell_b+\nabla \partial^{\alpha}h_b.\label{Estimates_A_x_inter}
\end{eqnarray*}
We then multiply $\nabla a$ to both sides and use integrate by parts to see
\begin{eqnarray*}
\|\nabla \partial^{\alpha}_xa\|_{L^2_{x}}
&\leq&\|\partial^t \partial^{\alpha}b\|_{L^2_{x}}+\|\partial^{\alpha}\ell_b\|+\|\partial^{\alpha}h_b\|_{L^2_{x}}\cr
&\leq&\|\partial_x\partial_t\partial^{\alpha-1}b\|_{L^2_{x}}+\|\partial^{\alpha}\ell_b\|+\|\partial^{\alpha}h_b\|_{L^2_{x}}\cr
&\leq& \sum_{|\bar\alpha|\leq N-1}\|\partial^{\bar\alpha}\tilde\ell\|_{L^2_x}
+\sum_{|\bar\alpha|\leq N-1}\|\partial^{\bar\alpha}\tilde h\|_{L^2_x}.
\end{eqnarray*}
In the last line, we employed Lemma \ref{EllipticEstimate_b}.\newline
(3) follows from Poincare inequality combined with (2) and the conservation of mass:
\begin{eqnarray*}
\|a\|_{L^2_x}\leq \|\nabla a\|_{L^2_x}
\end{eqnarray*}
\end{proof}
\begin{lemma}\label{Estimates_ELL}For $|\alpha|\leq N-1$, we have
\begin{eqnarray*}
&&(1) \sum_{|\alpha|\leq N-1}
\|\partial^{\alpha}\ell\|_{L^2_{x}}\leq
C\sum_{|\gamma|\leq N} \|(I-\tilde{P})\partial^{\alpha} f\|_{L^2_{x,v}},\cr
&&(2) \sum_{|\alpha|\leq N}
\|\partial^{\alpha} h\|_{L^2_{x}}\leq
C\big(\sqrt{M_0}+M_0\big)\sum_{|\alpha|\leq N} \|\partial^{\alpha} f\|_{L^2_{x,v}}.
\end{eqnarray*}
\end{lemma}
\begin{proof}
(1) Note that there exists constants $\lambda_n$ such that $\partial^{\alpha}\ell$ takes the following form
\begin{eqnarray*}
\sum^{13}_{n=1}\lambda^{n}\int_{\bbr^3}\partial^{\alpha}\ell\big\{(I-\tilde{P})f\big\}\cdot\varepsilon_n(v)dv,
\end{eqnarray*}
where $\varepsilon_n$ denotes the orthogonal basis for the 13 dimensional space spanned by
\[
\{\sqrt{m}, v_i\sqrt{m}, v_i^2\sqrt{m}, v_iv_j\sqrt{m}, |v|^2v_i\sqrt{m}~|~1\leq i,j\leq 3\}.
\]
We then observe that
\begin{eqnarray*}
&&\|\partial^{\alpha}\ell(\{I-\tilde{P}\}f)\cdot \varepsilon_n(v)dv\|^2_{L^2_x}\cr
&&\hspace{0.5cm}=\Big\|\int\big(-\{\partial_t+v\cdot \nabla+L\}(I-\tilde{P})\partial^{\alpha}f\big)\cdot\varepsilon_n(v)dv\Big\|_{L^2_x}\cr
&&\hspace{0.5cm}\leq\int|\varepsilon_n(v)|dv\times\cr
&&\hspace{0.5cm}\int|\varepsilon_n(v)|\Big\{|(I-\tilde{P})\partial^0\partial^{\alpha}f|^2
+|v|^2|(I-\tilde{P})\nabla_x\partial^{\alpha}f|^2+|(L(I-\tilde{P})\partial^{\alpha}f|^2\Big\}dxdv\cr
&&\hspace{0.5cm}\leq C\Big\{\|(I-\tilde{P})\partial^0\partial^{\alpha}f\|_{L^2_{x,v}}+\|(I-\tilde{P})\nabla\partial{\alpha}f\|_{L^2_{x,v}}
+\|(I-\tilde{P})\partial^{\alpha}f\|_{L^2_{x,v}}\Big\}^2.
\end{eqnarray*}
This completes the proof of (1).\newline
(2) As in (1), terms in $\partial^{\alpha} h$ can be presented as
\begin{eqnarray*}
\sum^{13}\bar\lambda^{n}\int_{\bbr^3}\partial^{\alpha}\Gamma(f,f)\cdot\varepsilon_n(v)dv
\end{eqnarray*}
for some constants $\bar\lambda_n$
We now apply Lemma \ref{bilinear.estimate} (3) to get
\begin{eqnarray*}
\Big\|\int_{\bbr^3}\partial^{\alpha}\Gamma(f)\cdot \varepsilon_n(v)dv\Big\|_{L^2_x}
&\leq&\sum_{|\alpha_1|+|\alpha_2|\leq|\alpha|}
\Big\|\int_{\bbr^3}\Gamma_{1,2,3}(\partial^{\alpha_1}f,\partial^{\alpha_2}f)\cdot \varepsilon_n(v)dv\Big\|_{L^2_x}\cr
&+&\sum_{\substack{|\alpha_1|+|\alpha_2|+|\alpha_3|\\\leq|\alpha|}}
\Big\|\int_{\bbr^3}\Gamma_{4}(\partial^{\alpha_1}f,\partial^{\alpha_2}f,\partial^{\alpha_2}f)\cdot \varepsilon_n(v)dv\Big\|_{L^2_x}\cr
&\leq&C\big(\sqrt{M}+M\big)\sum_{|\alpha|\leq N}\|\partial^{\alpha}f\|_{L^2_{x,v}}.
\end{eqnarray*}
This completes the proof.
\end{proof}
We can now prove the main theorem of this section.
%
%
%
%
\begin{theorem} \label{coercivity2}Let $f$ be a classical solution of (\ref{main.1}). Then there exists
$M$ and $\delta=\delta(M)$ such that if
\[
\sum_{|\alpha|\leq N}\|\partial^{\alpha}f(t)\|^2_{L^2_{x,v}}\leq M,
\]
then
There exists $\delta>0$ such that
\begin{equation*}
\sum_{|\alpha|\leq N}\langle L\partial^{\alpha} f,\partial^{\alpha} f\rangle\leq-\delta
\sum_{|\alpha|\leq N}\|\partial^{\alpha}f\|^2_{L^2_{x,v}}.
\end{equation*}
\end{theorem}
\begin{proof}
By Lemma \ref{EllipticEstimate_b}~-~\ref{Estimate_A1}, we have
\begin{eqnarray*}
\sum_{|\alpha|\leq N}\|\partial^{\alpha}a\|_{L^2_{x}}+
\|\partial^{\alpha}b\|_{L^2_{x}}+\|\partial^{\alpha}c\|_{L^2_{x}}
\leq\sum_{|\alpha|\leq N-1}\|\partial^{\alpha}\ell\|_{L^2_{x,v}}+\sum_{|\alpha|\leq N}\|\partial^{\alpha} h\|_{L^2_{x,v}}.
\end{eqnarray*}
We then apply Lemma \ref{Estimates_ELL} to see
\begin{eqnarray*}
\sum_{|\alpha|\leq N}\|\partial^{\alpha}\tilde{P}f\|_{L^2_{x,v}}&\leq&
C\sum_{|\alpha|\leq N}\|\partial^{\alpha}a\|_{L^2_{x}}+
\|\partial^{\alpha}b\|_{L^2_{x}}+\|\partial^{\alpha}c\|_{L^2_{x}}\cr
&\leq& C\sum_{|\alpha|\leq N} \|(I-\tilde{P})\partial^{\alpha} f\|_{L^2_{x,v}}
+C\sqrt{M_0}\sum_{|\alpha|\leq N} \|\partial^{\alpha} f\|_{L^2_{x,v}}.
\end{eqnarray*}
Hence we have from Lemma \ref{coercivity1} and the equivalence estimate (\ref{eqivalence})
\begin{eqnarray*}
\sum_{|\alpha|\leq N}\langle L\partial^{\alpha} f,\partial^{\alpha} f\rangle
&=&-\nu_c\sum_{|\alpha|\leq N}\|(I-P)\partial^{\alpha}f\|^2_{L^2_{x,v}}\cr
&\leq&-\nu_cC\sum_{|\alpha|\leq N}\|(I-\tilde{P})\partial^{\alpha}f\|^2_{L^2_{x,v}}\cr
&\leq&-\nu_c C_1\Big\{\sum_{|\alpha|\leq N}\|\partial^{\alpha}\tilde{P}f\|^2_{L^2_{x,v}}-C_2\sqrt{M}\sum_{|\alpha|\leq N}
\|\partial^{\alpha}f\|^2_{L^2_{x,v}}\Big\}\cr
&\leq&-\frac{\min\{\nu_c, \nu_c C_1\}}{2}
\Big\{\sum_{|\alpha|\leq N}\|\partial^{\alpha}\tilde{P}f\|^2_{L^2_{x,v}}+\|\partial^{\alpha}(1-\tilde{P})f\|^2_{L^2_{x,v}}\Big\}\cr
&&+\frac{1}{2}\nu_c C_2\sqrt{M}\sum_{|\alpha|\leq N}
\|\partial^{\alpha}f\|^2_{L^2_{x,v}}\cr
&\leq&-\frac{\nu_c}{2}\Big\{\min\{1,C_1\}-C_2\sqrt{M}\hspace{0.1cm}\Big\}\sum_{|\alpha|\leq N}\|\partial^{\alpha}f\|^2_{L^2_{x,v}}.
\end{eqnarray*}
We then choose $M$ sufficiently small such that
\[
\min\{1,C_1\}>C_2\sqrt{M}
\]
to obtain the desired result.
\end{proof}
\section{Proof of the main theorem}
In this section, we derive a refined energy estimate for the Boltzmann-BGK model and establish the main result.
Let $f$ be the unique smooth solution constructed in Theorem \ref{local.existence}.
First we take $\partial^{\alpha}$ on both sides of (\ref{Linearized_Main}) to have
\begin{eqnarray*}
[\partial_t+v\cdot \nabla +L]\partial^{\alpha}f=\partial^{\alpha}\Gamma(f).
\end{eqnarray*}
We multiply $\partial^{\alpha}f$, integrate over $\bbt^d\times\bbr^d$ and apply Lemma \ref{bilinear.estimate} (1)
and Theorem \ref{coercivity2} to obtain
\begin{eqnarray*}
E^{\alpha}:\quad\frac{1}{2}\frac{d}{dt}\|\partial^{\alpha}f\|^2_{L^2_{x,v}}+\delta\|\partial^{\alpha}f\|^2_{L^2_{x,v}}\leq
C\sqrt{E(t)}|||f|||^2.
\end{eqnarray*}
For $\beta\neq0$, we take $\partial^{\alpha}_{\beta}$ to obtain
\begin{eqnarray}\label{partial.alpha.beta}
[\partial_t+v\cdot \nabla+\nu_c]\partial^{\alpha}_{\beta} f&=&
-\sum_{i}\partial^{\alpha+\bar{e}_i}_{\beta-e_i}f^{n+1}
+\partial_{\beta}P\partial^{\alpha}f
+\partial^{\alpha}_{\beta}\Gamma(f),
\end{eqnarray}
where $e_1=(1,0,0), e_2=(0,1,0), e_3=(0,0,1)$ and  $\bar e_1=(0,1,0,0), \bar e_2=(0,0,1,0), \bar e_3=(0,0,0,1)$.
and we used the following relation:
\begin{eqnarray*}
\partial^{\alpha}_{\beta}(v\cdot \nabla_xf)&=&\partial^{\alpha}_{\beta}\Big\{\sum_{1\leq i\leq 3}v_i\partial^{e_i}f\Big\}\cr
&=&v\cdot \nabla_x\partial^{\alpha}_{\beta}f+\sum_{i}\sum_{\bar\beta\neq0}\partial_{\bar\beta}v_i\partial^{\alpha}_{\beta-\bar\beta}\partial^{e_i}f\cr
&=&v\cdot \nabla_x\partial^{\alpha}_{\beta}f+\sum_{i}\partial^{\alpha+\bar e_i}_{\beta-e_i}f.
\end{eqnarray*}
Multiplying  $\partial^{\alpha}_{\beta}f$ to both sides of (\ref{partial.alpha.beta}) and integrating over $\bbt^d\times\bbr^d$, we obtain by an almost identical manner
as in the local existence case
\begin{eqnarray*}
&&\frac{1}{2}\frac{d}{dt}\|\partial^{\alpha}_{\beta}f\|^2_{L^2_{x,v}}+\nu_c\|\partial^{\alpha}_{\beta}f\|^2_{L^2_{x,v}}\cr
&&\hspace{1cm}\leq -\sum_i\langle\partial^{\alpha+\bar{e}_i}_{\beta-e_i}f,\partial^{\alpha}_{\beta}f\rangle
+\nu_c\langle\partial_{\beta}P\partial^{\alpha}f,\partial^{\alpha}_{\beta}f\rangle
+\langle\partial^{\alpha}_{\beta}\Gamma(f),\partial^{\alpha}_{\beta}f\rangle\cr
&&\hspace{1cm}\leq \sum_iC_{\varepsilon}\|\partial^{\alpha+\bar{e}_i}_{\beta-e_i}f\|^2_{L^2_{x,v}}
+\varepsilon\|\partial^{\alpha}_{\beta}f\|^2_{L^2_{x,v}}
+C_{\varepsilon}\|\partial^{\alpha}f\|^2_{L^2_{x,v}}+
\varepsilon\|\partial^{\alpha}_{\beta}f\|^2_{L^2_{x,v}}\cr
&&\hspace{1cm}+ C\sum_{\substack{|\alpha_1|+|\alpha_2|\leq |\alpha|}}\int_{\bbr^d}\|\partial^{\alpha_1}f\|_{L^2_{x,v}}
\|\partial^{\alpha_2}f\|_{L^2_{x,v}}\|\partial^{\alpha}_{\beta}f\|_{L^2_{x,v}}dx\cr
&&\hspace{1cm}+ C\sum_{\substack{|\alpha_1|+|\alpha_2|\leq |\alpha|\\|\beta_2|\leq |\beta|}}\int_{\bbr^d}\|\partial^{\alpha_1}f\|_{L^2_{x,v}}
\|\partial^{\alpha_2}_{\beta_2}f\|_{L^2_{x,v}}\|\partial^{\alpha}_{\beta}f\|_{L^2_{x,v}}dx\cr
&&\hspace{1cm}+ C\sum_{\substack{|\alpha_1|+|\alpha_2|+|\alpha_3|\\\leq |\alpha|}}\int_{\bbr^d}\int_{\bbr^d}\|\partial^{\alpha_1}f\|_{L^2_{x,v}}
\|\partial^{\alpha_2}f\|_{L^2_{x,v}}\|\partial^{\alpha_3}f\|_{L^2_{x,v}}
\|\partial^{\alpha}_{\beta}f\|_{L^2_{x,v}}dx.
\end{eqnarray*}
Therefore, we have for sufficiently small $\varepsilon$
\begin{eqnarray*}
&&\frac{d}{dt}\|\partial^{\alpha}_{\beta}f\|^2_{L^2_{x,v}}+\frac{\nu_c}{2}\|\partial^{\alpha}f\|^2_{L^2_{x,v}}\cr
&&\hspace{1cm}\leq C_{\varepsilon}\sum_i\|\partial^{\alpha+\bar{e}_i}_{\beta-e_i}f\|^2_{L^2_{x,v}}
+C_{\varepsilon}\|\partial^{\alpha}f\|^2_{L^2_{x,v}}\cr
&&\hspace{1cm}+ C\sum_{\substack{|\alpha_1|+|\alpha_2|\leq |\alpha|}}\int_{\bbr^d}\|\partial^{\alpha_1}f\|_{L^2_{x,v}}
\|\partial^{\alpha_2}f\|_{L^2_{x,v}}\|\partial^{\alpha}_{\beta}f\|_{L^2_{x,v}}dx\cr
&&\hspace{1cm}+ C\sum_{\substack{|\alpha_1|+|\alpha_2|\leq |\alpha|\\|\beta_2|\leq |\beta|}}\int_{\bbr^d}\|\partial^{\alpha_1}f\|_{L^2_{x,v}}
\|\partial^{\alpha_2}_{\beta_2}f\|_{L^2_{x,v}}\|\partial^{\alpha}_{\beta}f\|_{L^2_{x,v}}dx\cr
&&\hspace{1cm}+ C\sum_{\substack{|\alpha_1|+|\alpha_2|+|\alpha_3|\\\leq |\alpha|}}\int_{\bbr^d}\int_{\bbr^d}\|\partial^{\alpha_1}f\|_{L^2_{x,v}}
\|\partial^{\alpha_2}f\|_{L^2_{x,v}}\|\partial^{\alpha_3}f\|_{L^2_{x,v}}
\|\partial^{\alpha}_{\beta}f\|_{L^2_{x,v}}dx\cr
&&\hspace{1cm}\equiv C_{\varepsilon}\sum_i\|\partial^{\alpha+\bar e_i}_{\beta-e_i}f\|^2_{L^2_{x,v}}
+C_{\varepsilon}\|\partial^{\alpha}f\|^2_{L^2_{x,v}}+I_{np}.
\end{eqnarray*}
The estimates for $I$ can be treated almost identically as in the proof of Theorem
\ref{local.existence} to obtain
\begin{eqnarray*}
I_{np}\leq C\big\{\sqrt{E(t)}+E(t)\big\}|||f|||^2,
\end{eqnarray*}
where we used the Sobolev embedding $H^s\hookrightarrow L^{\infty}.$
This yields
\begin{eqnarray*}
E^{\alpha}_{\beta}:\quad\frac{1}{2}\frac{d}{dt}\|\partial^{\alpha}_{\beta}f\|^2_{L^2_{x,v}}
+\frac{\nu_c}{2}\|\partial^{\alpha}_{\beta}f\|^2_{L^2_{x,v}}
&\leq&C_{\varepsilon}\sum_i\|\partial^{\alpha+\bar e_i}_{\beta-e_i}f\|^2_{L^2_{x,v}}
+C_{\varepsilon}\|\partial^{\alpha}f\|^2_{L^2_{x,v}}\cr
&+& C\big\{\sqrt{E(t)}+E(t)\big\}|||f|||^2.
\end{eqnarray*}
From the above inequality, we observe that the bad terms of $\sum_{|\beta|=m+1}E^{\alpha}_{\beta}$, that is
\[
\sum_{|\beta|=m+1}\Big\{C_{\varepsilon}\sum_i\|\partial^{\alpha+\bar e_i}_{\beta-e_i}f\|^2_{L^2_{x,v}}
+C_{\varepsilon}\|\partial^{\alpha}f\|^2_{L^2_{x,v}}\Big\},
\]
can be absorbed in the good terms of $C_m\sum_{|\beta|=m}E^{\alpha}_{\beta}+C_m\sum_{\alpha}E^{\alpha}$ if $C_m$
is sufficiently large. Therefore, by an induction argument, we can find constants $\bar{C}_m$ and $\delta_m$ such that
\begin{eqnarray*}
\sum_{\substack{|\alpha|+|\beta|\leq N,\\ |\beta|\leq m}}
\Big\{\bar{C}_m\frac{d}{dt}\|\partial^{\alpha}_{\beta}f\|^2_{L^2_{x,v}}
+\delta_m\|\partial^{\alpha}_{\beta}f\|^2_{L^2_{x,v}}\Big\}\leq C_N\big\{\sqrt{E(t)}+E(t)\big\}|||f|||^2.
\end{eqnarray*}
We now suppose $E<1$ without loss of generality and set $m=N$ to obtain
\begin{eqnarray*}
\sum_{\substack{|\alpha|+|\beta|\leq N}}
\Big\{\bar{C}_N\frac{d}{dt}\|\partial^{\alpha}_{\beta}f\|^2_{L^2_{x,v}}
+\delta_N\|\partial^{\alpha}_{\beta}f\|^2_{L^2_{x,v}}\Big\}\leq C_N\big\{\sqrt{E(t)}\big\}|||f|||^2.
\end{eqnarray*}
Notice that we used $E(t)\leq \sqrt{E(t)}$ and redefined $2C_N$ by $C_N$.
We now define $y(t)$ as
\[
y(t)=\sum_{|\alpha|+|\beta|\leq N}
\bar{C}_N\frac{d}{dt}\|\partial^{\alpha}_{\beta}f\|^2_{L^2_{x,v}}.
\]
We choose a constant $C_1$ such that
\begin{eqnarray*}
\frac{1}{C_1}\Big\{y(t)+\frac{\delta_N}{2}\int^t_0|||f(s)|||^2ds\Big\}\leq E(t)\leq C_1\Big\{y(t)+\frac{\delta_N}{2}\int^t_0|||f(s)|||^2ds\Big\}
\end{eqnarray*}
We define
\[
M=\min\Big\{\frac{\delta^2_N}{8C_N^2C_1^2},\frac{M_0}{2C^2_2}\Big\}
\]
and choose the initial data sufficiently small in the sense that
\[
E(0)\leq M<M_0.
\]
Let $T>0$ be given as
\[
T=\sup_t\Big\{t: E(t)\leq 2C^2_1 M\Big\}>0,
\]
which gives
\[
E(t)\leq 2C^2_1M\leq M_0.
\]
We then have for $0\leq t\leq T$
\begin{align}
\begin{aligned}\label{y+|||f|||}
y^{\prime}(t)+\delta_N|||f|||^2(t)&\leq C_N\sqrt{E(t)}|||f|||^2(t)\cr
&\leq C_NC_1\sqrt{2M}|||f|||^2(t)\cr
&\leq\frac{\delta_N}{2}|||f|||^2(t).
\end{aligned}
\end{align}
Therefore, integration above over $0\leq t\leq T$ yields
\begin{eqnarray*}
E(t)&\leq& C_1\Big\{y(t)+\frac{\delta_N}{2}\int^t_0|||f(s)|||^2ds\Big\}\cr
&\leq& C_1y(0)\cr
&\leq&C^2_1E(0)\cr
&\leq&C^2_1M\cr
&<&2C^2_1M.
\end{eqnarray*}
This is a contradiction considering the continuity of $E$ and the definition of $T$.
Hence we have $T=\infty$. 
By (\ref{y+|||f|||}) and $y(t)\leq C|||f(t)|||$ for some constant $C$, we have
\begin{eqnarray*}
y^{\prime}(t)+\frac{\delta_N}{2}y(t)\leq y^{\prime}(t)+\frac{\delta_N}{2}|||f|||^2(t)\leq 0,
\end{eqnarray*}
which gives the exponential decay of the perturbation.
\newline We are now left with the $L^2$-stability estimate.
Let $\bar f$ be another solution corresponding to initial data $\bar f_0$. We subtract the equation for $\bar f$ from
the equation for $f$ to see
\begin{eqnarray*}\label{Uniqueness2}
\big\{\partial_t+v\cdot \nabla\big\}(f-\bar f)&=&L(f-\bar f)+\Gamma_{1,2,3}(f-\bar f,f)+\Gamma_{1,2,3}(\bar f,f-\bar f)\cr
&+&\Gamma_4(f-\bar f,f,f)+\Gamma_4(\bar f,f-\bar f,f)+\Gamma_4(\bar f,\bar f,f-\bar f).
\end{eqnarray*}
We then multiply $f-\bar f$ and integrate over $x,v$ and $t$ to have
\begin{eqnarray*}
\frac{1}{2}\frac{d}{dt}\|f(t)-\bar f(t)\|^2_{L^2_{x,v}}\!\!\!&=&\!\langle L(f-\bar f),f-\bar f\rangle
+\langle\Gamma_{1,2,3}(f-\bar f,f)+\Gamma_{1,2,3}(\bar f,f-\bar f),f-\bar f\rangle\cr
&&+\Gamma_4(f-\bar f,f,f)+\Gamma_4(\bar f,f-\bar f,f)+\Gamma_4(\bar f,\bar f,f-\bar f),f-\bar f\rangle.
\end{eqnarray*}
We now apply the coercivity estimate in Theorem \ref{coercivity2}:
\[
\langle L(f-\bar f),f-\bar f\rangle\leq -\delta\|f-\bar f\|^2_{L^2_{x,v}}
\]
and the following estimates from Proposition \ref{bilinear.estimate} and (\ref{Sobolev}):
\begin{eqnarray*}
&&\langle\Gamma_{1,2,3}(f-\bar f,f)+\Gamma_{1,2,3}(\bar f,f-\bar f),f-\bar f\rangle\cr
&&\hspace{1cm}+\langle\Gamma_4(f-\bar f,f,f)+\Gamma_4(\bar f,f-\bar f,f)+\Gamma_4(\bar f,\bar f,f-\bar f),f-\bar f\rangle\cr
&&\hspace{1cm}\leq C\sum_{|\alpha|\leq 2}\big(\|\partial^{\alpha}f\|^2_{L^2_{x,v}}
+\|\partial^{\alpha}\bar f\|^2_{L^2_{x,v}}+\|\partial^{\alpha}f\|_{L^2_{x,v}}
+\|\partial^{\alpha}\bar f\|_{L^2_{x,v}}\big)\|f-\bar f\|^2_{L^2_{x,v}}\cr
&&\hspace{1cm}\leq C\Big\{\sqrt{E_f}(t)+\sqrt{E_{\bar f}}(t)
\Big\}\|f-\bar f\|^2_{L^2_{x,v}}
\end{eqnarray*}
to get
\begin{eqnarray*}
&&\frac{1}{2}\frac{d}{dt}\|f(t)-\bar f(t)\|^2_{L^2_{x,v}}+\Big(\delta-\Big\{\sqrt{E_f}(t)+\sqrt{E_{\bar f}}(t)\Big\}\Big)
\|f(t)-\bar f(t)\|^2_{L^2_{x,v}}\leq 0,
\end{eqnarray*}
which gives the desired result for sufficiently small $E_g(0)$ and $E_f(0)$.

\end{document}